\documentclass[a4paper,12pt]{article}

\usepackage[T1]{fontenc}
\usepackage{lmodern}
\usepackage[margin=1in]{geometry}
\usepackage{amsmath,amssymb,amsthm}
\usepackage{array,booktabs,enumitem,microtype,graphicx}
\usepackage{xcolor}
\usepackage{tikz}
\usepackage[hidelinks]{hyperref}
\usepackage[normalem]{ulem}
\newtheorem{proposition}{Proposition}
\newtheorem{corollary}[proposition]{Corollary}

\hypersetup{pdftitle={Risk-Constrained MPPI Tuning with Update-Aware \texttt{MA-ES} Racing},pdfsubject={Risk-constrained MPPI tuning with update-aware statistical racing},pdfkeywords={MPPI, \texttt{MA-ES}, statistical racing, chance constraints}}

\begin{document}

\begin{center}
	
	{\Large \bf Risk-Constrained MPPI Tuning\\with Update-Aware \texttt{MA-ES} Racing} 
	
	\vspace{0.5cm}

	{\large \bf Morteza Kimiaei$^{a}$, Roya Khalili-Amirabadi$^{b}$, 
		Mohsen Jalaeian-Farimani$^{c}$}

	{$^a$ \sl Faculty of Mathematics, University of Vienna, 1090 Vienna, Austria}
	\centerline{\sl Email: \href{mailto:kimiaeim83@univie.ac.at}{kimiaeim83@univie.ac.at}}
	\centerline{\sl WWW: \url{http://www.mat.univie.ac.at/~kimiaei}}
	\\
	{$^b$ \sl Department of Applied Mathematics, Ferdowsi University of Mashhad, Mashhad, Iran}
	\centerline{\sl Email:
	\href{mailto:roya.khalili.a@gmail.com}{roya.khalili.a@gmail.com}}\\
		{$^c$ \sl Department of Electronics, Information and Bioengineering (DEIB), Politecnico di Milano, Milan, Italy}
		\centerline{\sl Email:
	\href{mailto:mohsen.jalaeian@polimi.it}{mohsen.jalaeian@polimi.it}}
	\centerline{\sl WWW: \url{https://www.deib.polimi.it/eng/people/details/2012404}}
	\\

\end{center}

{\bf Abstract.}
We formulate offline tuning of Model Predictive Path Integral (\texttt{MPPI}) controllers as a chance-constrained black-box optimization problem and develop an update-aware framework based on Matrix Adaptation Evolution Strategy (\texttt{MA-ES}). We use statistical information from noisy controller evaluations to determine when the evolutionary update is sufficiently reliable. When the recombination direction is certified, a derivative-free nonmonotone line search (\texttt{NMLS}) is used to refine the candidate solution, and a separate validation stage assesses the final frozen controller. We establish a worst-case arithmetic-complexity bound for both online \texttt{MPPI} control and offline tuning. Numerical experiments on \texttt{TurtleBot~4}, \texttt{Jackal-like}, and \texttt{Ackermann} models show that the proposed \texttt{NMLS} refinement is competitive with standard \texttt{MA-ES}, transfers across the three vehicle models without model-specific retuning, and improves several cost and validation measures. Both \texttt{MA-ES} variants also compare favorably with the tested randomized derivative-free line search (\texttt{VRBBO}) configuration. The results support the practical applicability and transferability of the
proposed framework, while not claiming universal superiority or hardware-level safety.

\noindent\textbf{Keywords:} \texttt{MPPI}; \texttt{MA-ES}; nonmonotone derivative-free line search, chance-constrained controller tuning;
statistical racing; rank uncertainty; finite-sample validation.


\section{Introduction}

\label{sec:introduction}

Black-box optimization (BBO)  addresses problems in which objective-function values
can be obtained for candidate solutions, but useful analytic derivatives or
an explicit, exploitable model of the objective are unavailable. Such problems
arise naturally in numerous applications across science, engineering, industry,
and chemistry, particularly when objective evaluations require simulations,
physical experiments, or closed-loop executions rather than direct evaluations
of known analytic functions. Derivative-free optimization (DFO) methods provide
a general framework for addressing these problems by guiding the search using
function evaluations without requiring exact derivative information.
Comprehensive treatments of DFO methods, their theoretical foundations, and
practical applications can be found in the books by Audet and Hare~\cite{AudH}
and Conn et al.~\cite{ConSV}, as well as the survey by Larson
et al.~\cite{LarMW}. Furthermore, the works of
Mor{\'e} and Wild~\cite{MorW}, Rios and Sahinidis~\cite{RioS},
Kimiaei~\cite{VRDFON}, and Kimiaei and Neumaier~\cite{kimiaei2024madfo,kimiaei2021vrbbo,kimiaei2025matrs}
provide numerical investigations and comparative assessments of
DFO algorithms for problems with continuous and mixed-integer variables,
offering insights into their computational performance and practical behavior.

In many practical settings, black-box evaluations are also subject to
uncertainty. Deterministic inaccuracies may arise from modeling,
discretization, or numerical approximation, whereas stochastic variability
may result from simulation randomness, measurement noise, or other random
perturbations. Consequently, repeated evaluations of the same candidate
may produce different objective values, and an apparently superior observation
need not indicate a genuinely better underlying solution. These difficulties
are particularly pronounced when evaluations are computationally expensive
and only a limited evaluation budget is available. The central challenge is
therefore not merely to identify promising candidates efficiently, but also
to distinguish genuine improvements in objective performance from variability
in the observed function values, while making effective use of the available
evaluation budget.

In this black-box setting, the performance of a Model Predictive Path Integral
(\texttt{MPPI}) controller depends strongly on a small set of tuning parameters,
yet candidate settings can be judged only through noisy and potentially expensive
closed-loop trials. In our setting, stochastic variability arises from randomized
initial robot states, obstacle perturbations, observation and actuation
uncertainties, and the Gaussian sampling used by the \texttt{MPPI} controller.
Consequently, even with fixed controller parameters, different realizations
of these uncertainties can produce different closed-loop trajectories, task
costs, and safety outcomes. Our goal is to tune these parameters without
confusing random trial variation with genuine improvement, while keeping
safety explicit. We address this by combining full-trial evaluation,
uncertainty-aware evolutionary search, and independent final validation.

\texttt{MPPI} control evaluates sampled command sequences with a prediction
model and constructs a receding-horizon command update by importance-weighted
averaging. The parallel-computation formulation of Williams
et al.~\cite{williams2017} and the information-theoretic treatment of Williams
et al.~\cite{williams2018} established the foundations used here. Their
contributions concern the online optimal-control computation; they do not make
the cost weights, sampling covariance, or temperature immaterial to closed-loop
behavior. The question considered in this paper is how to select these
controller parameters from noisy, complete-task evaluations when the
probability of a specified safety failure is itself a design constraint.

Within this derivative-free setting, an evolution strategy iteratively updates
a probability distribution over candidate solutions through mutation, selection,
and recombination. Matrix Adaptation Evolution Strategy (\texttt{MA-ES}) is a
derivative-free evolution strategy in which the search distribution is adapted
through a global scale and a transformation matrix~\cite{beyer2017}. In each
generation, mutation generates random candidate perturbations from the current
search distribution, selection ranks the resulting candidates according to their
evaluated performance, and recombination forms a weighted combination of the
selected mutation directions. The resulting information updates the search mean,
while the evolution path, transformation matrix, and global scale adapt the
location, shape, and size of the distribution used to generate the next
population. Thus, \texttt{MA-ES} repeatedly alternates between generating
candidates, selecting the better ones, and adapting the search distribution.

To enhance the recombination step of this evolutionary search, we incorporate
ideas from derivative-free line-search methods. Such methods provide
globalization mechanisms for BBO, with specialized variants
designed to accommodate noisy function evaluations; representative examples
include \texttt{SDBOX}~\cite{LucS},
\texttt{VRBBO}~\cite{kimiaei2021vrbbo}, and
\texttt{VRDFON}~\cite{VRDFON}, with broader discussions given, for example, in
\cite[Section~2.3.4]{LarMW}. These methods evaluate trial
points along prescribed search directions and accept them according to
suitable decrease conditions. In fully derivative-free variants,
sufficient-decrease conditions can be formulated using positive forcing
functions without requiring directional-derivative estimates, thereby
avoiding the sensitivity of finite-difference approximations to noisy
function values. An important ingredient in several such methods is
\emph{extrapolation}: once a trial point is successful, the step length is
enlarged to explore farther along the same direction and potentially leave
flat or unfavorable regions more rapidly. Under noise, however, strict
monotone decrease can be unreliable because a single favorable or unfavorable
evaluation may be caused by perturbations rather than genuine progress.
Nonmonotone line-search strategies therefore compare new trial points with
a relaxed reference value rather than requiring decrease relative only to
the immediately preceding iterate; both deterministic and randomized
constructions have been studied extensively
\cite{LucS,AhoA,AmiEK,AmKK,iam:Birgin99,DinizEhrhardt2008,Grippo2014,grippo1986nonmonotone,Kim17,kimiaei2022new,KimR,kimiaeia2021trust,iam:Toint96}.
Motivated by these ideas, the present method augments the
\texttt{MA-ES} recombination phase with a derivative-free nonmonotone
line-search procedure along the selected recombination direction, allowing
both under-relaxed and extrapolated recombination steps.

\paragraph{Online proposal adaptation and command regularity.}
The connection between path-integral policy search and evolutionary
optimization predates the present application. Stulp and Sigaud's
PI$^2$-CMA adapts exploration covariance in Path Integral Policy Improvement,
and explicitly relates that policy-search procedure to the cross-entropy
method and covariance matrix adaptation~\cite{stulp2012}. PI$^2$-CMA is
not an outer-loop \texttt{MPPI} hyperparameter tuner, but it rules out claiming that
combining path-integral optimization with covariance adaptation is new.
More directly, Asmar et al.~\cite{asmar2023} formulate a joint distribution
over the command sequence and introduce adaptive importance sampling within
\texttt{MPPI}. Their Model Predictive Optimized Path Integral (\texttt{MPOPI}) strategies
include cross-entropy and covariance-matrix-adaptation evolution-strategy
updates of the online proposal. Keshavarz et al.~\cite{keshavarz2025}
subsequently investigate this \texttt{MPPI--CE--CMA} combination for whole-body
legged-robot control. These works optimize the distribution of commands
\emph{inside} each receding-horizon calculation, rather than selecting one
fixed vector of controller settings from complete navigation trials.

Other improvements address different properties of the online planner.
Smooth \texttt{MPPI} combines input lifting with an action-variation cost to suppress
chattering without an external smoothing filter~\cite{kim2022}.
Stein Variational Guided \texttt{MPPI} uses modified Stein variational updates to
identify a promising mode and guide the proposal when the target action
distribution is multimodal~\cite{honda2024}. Neither smoothing nor mode
selection is introduced here. We retain a specified Gaussian \texttt{MPPI} update,
including its density-ratio correction, so that the proposed statistical
evaluation mechanism is not confounded with a change of online solver.

\paragraph{Outer-loop tuning and adaptive controller parameters.}
An additional optimization layer around \texttt{MPPI} is also established prior art.
Goldfain's doctoral dissertation describes cost-function optimization on the
time scale of complete racing laps around an \texttt{MPPI} controller~\cite{goldfain2019}.
This provides an architectural precedent for an outer optimization loop. Ruhe et al.~\cite{ruhe2025}
develop Bayesian hyperparameter optimization using an Isaac Sim digital twin
and a ROS2/Nav2 implementation. Their author-maintained project description
explicitly includes batch size, control noise, and temperature among the
optimized settings~\cite{ruheproject}. Consequently, neither black-box \texttt{MPPI}
tuning nor jointly considering temperature and sampling parameters is a
sufficient novelty claim. The public \texttt{pytorch\_mppi} software also
provides \texttt{CMA-ES} and Bayesian-optimization-based autotuning, including
covariance, horizon, and temperature parameters~\cite{pytorchmppi}.
This is software prior art and is cited as such, rather than as a research
paper or evidence of comparative performance.

Adaptation can additionally be embedded in the control architecture.
Kim and Lee~\cite{kimlee2024} adjust \texttt{MPPI} cost weights from predicted states
in a multiconstraint guidance problem; their contribution is adaptive
weighting rather than \texttt{MA-ES}-based offline selection. Huang et
al.~\cite{huangpg2026} study advantage-actor--critic-guided \texttt{MPPI} parameter
adaptation for mobile robots. Their separate \texttt{ACW-MPPI}
study~\cite{huangacw2026} concerns a policy-embedded framework for real-time
mobile-robot navigation. Beyond \texttt{MPPI}-specific adaptation, reinforcement-learning-based
approaches have also been investigated for unmanned ground
vehicle navigation. For example, Farimani et al.~\cite{11108832}
developed a fuzzy Q-learning framework incorporating a
fuzzified Bellman equation for UGV navigation and path planning.
These studies illustrate the broader applicability of
learning-based methods to robotic navigation and control.
Our contribution, however, focuses on the risk-aware statistical
evaluation and ranking of fixed \texttt{MPPI} parameter configurations
and their influence on evolutionary search updates.
More closely related to \texttt{MPPI}, the \texttt{RC-MPPI} preprint
of Yoon and Kim~\cite{yoon2026} links online constraint
tightening, safety-cost shaping, and temperature relaxation
to prediction--execution residuals. In contrast, our framework
keeps the \texttt{MPPI} controller parameters fixed throughout each
trial and during deployment, while adapting only the offline
evaluation, ranking, and selection mechanisms.

\paragraph{Risk-aware planning and statistical validation.}
Risk treatment within \texttt{MPPI} must be distinguished from the reliability of
an offline tuning decision. Yin et al.~\cite{yin2023} integrate Conditional Value-at-Risk
(\texttt{CVaR}) into \texttt{MPPI} to account for unfavorable trajectory
outcomes under uncertainty. \texttt{CVaR} concerns the tail of a loss distribution,
whereas the chance constraint studied here concerns the probability of a
specified binary failure event over an entire closed-loop trial. Trevisan et al.~\cite{trevisan2025}
approximate joint collision probabilities across dynamic obstacles with
Monte Carlo sampling for navigation in crowds. The \texttt{DUCCT-MPPI} preprint of
Serfling et al.~\cite{serfling2026} combines robot-state and dynamic-object
prediction uncertainty and evaluates collision-risk calibration using
proper scoring rules. These contributions establish that risk-aware \texttt{MPPI},
Monte Carlo collision-risk evaluation, and calibration assessment already
exist. Our confidence statements concern repeated \emph{complete-trial}
outcomes of frozen candidate controllers, not the correctness of an online
collision predictor. They are conditional on the stated scenario law and
do not replace robust control, a safety filter, or a model-validity argument.

\paragraph{Constrained black-box search and statistical racing.}
The outer optimization ingredients likewise have independent precedents.
Beyond the established \texttt{MA-ES} framework~\cite{beyer2017},
feasibility-first evolutionary selection is established in constrained
optimization~\cite{deb2000}, and constrained Bayesian optimization models
both an expensive objective and expensive constraints~\cite{gardner2014}.
Most importantly, Heidrich-Meisner and Igel~\cite{heidrich2009} already
combine \texttt{CMA-ES} policy search with Hoeffding and Bernstein races to adapt
the number of policy evaluations needed for reliable selection. Therefore,
adaptive resampling, confidence-based evolutionary ranking, and using an
\texttt{ES} for noisy control-policy evaluation are not claimed as new here.

Hansen et al.~\cite{hansen2009} quantify uncertainty through population
rank changes and apply uncertainty-handling \texttt{CMA-ES} to experimental feedback
controller tuning. More recently, Wang and Lu~\cite{wang2026} replace hard
rank weights with conditional expected weights, estimated by residual
bootstrapping, to trade ranking fidelity against optimization depth. Their
ICML 2026 study provides a particularly relevant fixed-budget comparison
point because it argues that excessive within-generation denoising can
reduce the number of useful optimizer updates. Thus, linking rank
uncertainty to an evolutionary update is not itself a new principle; the
specific object bounded here is the discrepancy of a hard-rank,
risk-constrained \texttt{MA-ES} update.

\paragraph{Advanced matrix-adaptation methods for noisy BBO.}
Beyond the basic \texttt{MA-ES} framework, more elaborate matrix-adaptation
methods have been developed specifically for noisy DFO. Kimiaei and Neumaier~\cite{kimiaei2024madfo} proposed
\texttt{MADFO}, an extension of fast matrix-adaptation evolution strategies
designed for noisy continuous unconstrained optimization. In addition to the
usual mutation, selection, and recombination phases, \texttt{MADFO} modifies
the generation of mutation and recombination steps to avoid steps that are
either excessively small or excessively large, constructs recombination
subspace directions from current and previous search information, and
supplements recombination by a derivative-free nonmonotone line search.
In particular, extrapolation is performed along a recombination direction in
an attempt to identify a point with a lower inexact objective value than the
standard recombination point. A randomized nonmonotone reference value is
used to reduce sensitivity to spuriously favorable noisy evaluations, while
additional triangle-subspace sampling provides a fallback when the line search
does not identify a satisfactory new point. Thus, \texttt{MADFO} preserves the
rank-based matrix-adaptation mechanism while augmenting the search for useful
mutation and recombination points with derivative-free globalization
strategies appropriate for noisy objectives.

A further development is the matrix-adaptation trust-region strategy
\texttt{MATRS} of Kimiaei and Neumaier~\cite{kimiaei2025matrs}, which targets
noisy bound-constrained DFO with mixed-integer variables.
\texttt{MATRS} combines matrix adaptation with derivative-free line searches,
trust-region steps, and dedicated mixed-integer searches. Its main cycle
contains mutation, selection, recombination, trust-region, and mixed-integer
phases. Line searches are performed along mutation and recombination directions
and, when appropriate, along their opposite directions; extrapolation is used
to search beyond an initially successful trial point. If these searches do not
produce a new best point, trust-region models provide additional trial
directions, while a separate mixed-integer phase handles discrete and
continuous variables jointly. The method also employs space-filling directions
and admits a reentrant implementation intended for expensive algorithm-tuning
problems in which objective values are supplied by an external simulation or
oracle. Hence, \texttt{MADFO} and \texttt{MATRS} demonstrate that substantially
more sophisticated search and globalization mechanisms can be built around
matrix adaptation for noisy black-box objectives. However, the numerical
studies reported for these methods concern generic derivative-free benchmark
problems rather than robotic closed-loop control or navigation. To the best of
our knowledge, these advanced matrix-adaptation variants have not yet been
investigated for tuning an \texttt{MPPI} controller under repeated robotic
closed-loop trials. Moreover, neither method addresses the particular issue
considered here: converting uncertainty in full-trial failure probabilities
and costs into a quantitative bound on the resulting risk-constrained
\texttt{MA-ES} update.

\paragraph{Specific research gap.}
The remaining design question addressed in this paper is narrower: how can
uncertainty about both full-trial failure probability and task cost be
translated into a quantitative bound on the \emph{actual \texttt{MA-ES} update}
used to tune \texttt{MPPI}? An empirical feasibility-first ordering does not answer
this question. Sampling error near the failure threshold can move a
candidate across the feasibility boundary; uncertainty within the elite
set also changes a rank-weighted update even when elite membership is
unchanged. Unlike conditional averaging over noisy ranks, the proposed
rule retains a hard-rank update and bounds its discrepancy from the true
\emph{risk-constrained} population ordering, uniformly over declared looks.
The proposed construction connects simultaneous statistical
comparisons to admissible rank ranges and then to perturbation bounds for
the \texttt{MA-ES} mean, evolution path, transformation matrix, and logarithmic
scale. Evaluation can stop when those update bounds meet prescribed
tolerances, without resolving comparisons that do not materially affect
the update. The rank-to-update argument is agnostic to the particular
confidence interval used: the finite-look Clopper--Pearson/Hoeffding
construction below is one conservative instantiation, and any tighter
replacement must preserve the stated simultaneous-coverage property. This
is an explicit synthesis and analysis of established building blocks, not
a claim that any individual building block, all \texttt{MPPI} tuning, or all
evolutionary racing is new.

\paragraph{Contributions and scope.}
We develop an update-aware, risk-constrained \texttt{MA-ES} framework for
tuning \texttt{MPPI} controllers from noisy complete-task evaluations. The
overall procedure is as follows. Candidate controllers are evaluated under a
chance-constrained full-trial criterion; statistical confidence bounds are
used to determine which population orderings remain admissible; these
orderings are propagated to the corresponding \texttt{MA-ES} update
quantities; and, when the resulting recombination direction is certified, a derivative-free nonmonotone line-search refinement is applied before the next
population is generated. The final frozen controller is then assessed by an
independent fixed-sample validation rule.

The main features are:
\begin{itemize}
	\item A full-trial chance-constrained formulation jointly tunes the
	\texttt{MPPI} planning weights, temperature, and sampling scales while
	keeping the online computational budget fixed and evaluating performance
	with an external task metric that is distinct from the tuned planning
	score.
	
	\item A finite-look statistical evaluator combines exact binomial confidence
	intervals with bounded-cost comparisons to construct a
	confidence-supported partial ordering of the population. The admissible
	orderings are then mapped to explicit bounds on the standard
	\texttt{MA-ES} update quantities, yielding an update-aware stopping
	criterion for the rank-based evolutionary recursion.
	
	\item Once the recombination direction is rank-certified, we refine the
	nominal \texttt{MA-ES} recombination point by a derivative-free
	nonmonotone line search adapted from the recombination line-search strategy
	of Kimiaei and Neumaier~\cite{kimiaei2024madfo}. The additional
	randomization in the original nonmonotone reference is removed and replaced
	by a deterministic mixing rule. The method evaluates extrapolated points
	along the certified recombination direction and uses a derivative-free
	forcing condition together with the outer risk-first comparison to select
	either the nominal recombination point or an accepted extrapolated point.
	No derivative information or additional random direction reversal is used.
	
	\item A separate fixed-sample validation stage controls the
	false-acceptance probability of the final frozen controller. The binomial
	and concentration tools are classical~\cite{clopper1934,hoeffding1963};
	the contribution lies in their integration with the chance-constrained
	\texttt{MPPI} tuning problem, the risk-to-rank-to-\texttt{MA-ES}-update
	connection, and the certified derivative-free refinement of the
	recombination point.
	
	\item A worst-case arithmetic-complexity analysis quantifies the cost
	of online \texttt{MPPI} updates and complete closed-loop trials,
	as well as the total offline tuning cost, including matrix adaptation,
	statistical racing, and nonmonotone line-search evaluations.
	The resulting bounds explicitly characterize the dependence on
	the sampling budget, prediction horizon, and optimizer dimensions,
	without asserting real-time feasibility or an optimization
	convergence rate.
\end{itemize}

Unlike the randomized nonmonotone line-search scheme of Kimiaei and
Neumaier~\cite{kimiaei2024madfo}, developed for noisy DFO, the present approach employs a deterministic construction
of the nonmonotone reference. This choice does not assume that the closed-loop objective evaluations
have negligible noise. Rather, stochasticity is already present in the
scenario evaluations and the mutation mechanism of \texttt{MA-ES}, while
common random numbers are used to improve the consistency of candidate
comparisons. By fixing the nonmonotone parameters and updating the
reference through a deterministic memory rule, we avoid introducing an
additional source of algorithmic randomization and obtain a reference construction that is reproducible conditional on
the evaluated population costs. This provides a controlled setting for assessing
the contribution of the nonmonotone recombination line search under a
limited evaluation budget.

The rank-to-update analysis applies directly to the standard \texttt{MA-ES}
quantities induced by the admissible population orderings. When the line
search is activated, rank certification guarantees that the chosen and true
orderings produce the same recombination direction; conditional on the
realized population and stored scenarios, the subsequent nonmonotone
reference, extrapolation rule, and final selection are deterministic.
Therefore, the line-search refinement introduces no additional
rank-induced discrepancy, while the evolution-path, transformation-matrix,
and scale updates remain those of the standard \texttt{MA-ES} recursion.

We provide the construction, assumptions, proofs, and numerical studies on
the three kinematic vehicle models considered in this paper. The
update-aware study evaluates independent optimization seeds against the
pre-tuning initialization, while a separate eleven-seed large-budget
experiment compares standard \texttt{MA-ES} with the tuned
\texttt{NMLS} variant. In that comparison the line search is activated after
every generation and is therefore an algorithmic ablation rather than a test
of the rank-certified activation rule. We do not report hardware
experiments, prove global optimization or closed-loop stability, or claim
universal empirical superiority; the numerical conclusions are restricted
to the declared models, budgets, and evaluation protocol.

\paragraph{Organization of the paper.} The remainder of the paper is organized as follows.
We first formulate the risk-constrained controller-tuning problem in Section \ref{sec:probdef},
then present the proposed algorithm in Section \ref{sec:algorithm} and its mathematical analysis in Section \ref{sec:mathform}, including the statistical guarantees and computational complexity in Section \ref{sec:computational-complexity}.
Numerical experiments assess the framework across three kinematic robot models in Section \ref{sec:simulation}, followed by conclusions and limitations in Section \ref{sec:conclusion}.

With the motivation, prior work, contribution, and scope established, we next state the controller-tuning problem independently of the numerical procedure used to solve it.


\section{Problem definition and design choices}\label{sec:probdef}

Before constructing the tuning algorithm, we must specify exactly what is being optimized and what constitutes a safety failure. The goal of this section is to separate the controller-design problem from the numerical method used to solve it. We therefore define the complete-trial performance criterion, the safety requirement, and the controller parameters to be tuned.

Consider a robot governed by a parameterized controller $\pi_\theta$, with
controller parameters $\theta\in\Theta\subseteq\mathbb{R}^n$. The robot must
complete a task under variations in initial conditions, environments, sensing,
and motion. We represent all sources of randomness affecting a complete trial,
including environmental uncertainties and the internal sampling
of the controller, by a random scenario $\omega\sim\mathcal{D}$. A trial produces a trajectory and a scalar cost
$C(\theta,\omega)$. A separate failure indicator
$I_{\mathrm{fail}}(\theta,\omega)\in\{0,1\}$ records whether the robot
violated a safety requirement during the trial. For fixed controller
parameters $\theta$, both the observed cost and the failure indicator
may vary across independent realizations of $\omega$. Consequently,
the expected cost and failure probability are not directly available
to the optimization algorithm and must be estimated from finitely
many complete closed-loop trials.

The design problem is
\begin{equation}
\label{eq:problem}
\begin{aligned}
\underset{\theta\in\Theta}{\operatorname{minimize}}\quad
&J(\theta):=\mathbb{E}_{\omega\sim\mathcal{D}}[C(\theta,\omega)]\\
\text{subject to}\quad
&p_{\mathrm{fail}}(\theta)
:=\mathbb{P}_{\omega\sim\mathcal{D}}
  \bigl(I_{\mathrm{fail}}(\theta,\omega)=1\bigr)\leq\delta,
\end{aligned}
\end{equation}
where $\delta$ is the maximum acceptable failure probability. The global
minimizer of~\eqref{eq:problem} defines an ideal target, but a finite set of
robot trials does not generally identify or certify it. The operational goal
is to find a controller that meets stated performance and safety criteria on
independent trials. The formulation distinguishes the robotics problem from
the numerical procedure used to search for a solution.

Here, the controller is specialized to
$\pi_\theta=\pi^{\mathrm{\texttt{MPPI}}}_\theta$.
The vector $\theta$ contains planner cost weights, a temperature, and
sampling scales; it does not contain the robot's entire command sequence.
During a trial, $\theta$ is fixed and \texttt{MPPI} replans from each new state
estimate. Between trials, \texttt{MA-ES} changes $\theta$ according to measured
closed-loop performance. Thus the extension changes the realization of
the controller and its parameter mapping while retaining the outer
optimization problem and the standard rank-based \texttt{MA-ES} path,
matrix, and scale recursions; the recombination point may additionally
be refined by the line-search mechanism introduced below. The evaluation and selection steps will be
strengthened by statistically controlled, update-aware racing.

\subsection{Choice of the outer-loop optimizer}

Closed-loop costs can be nonsmooth or available only through simulation,
which motivates derivative-free search. \texttt{MA-ES} provides correlated search
directions through a full matrix without explicitly maintaining and
factorizing a separate covariance matrix~\cite{beyer2017}. Its storage is
$O(n^2)$, and its trial cost can dominate this algebra for low-dimensional
controller tuning. We use it as a concrete rank-weighted optimizer, not
because its superiority over \texttt{CMA-ES}, cross-entropy, or Bayesian optimization
has been established. The statistical comparison construction is not
exclusive to \texttt{MA-ES}; its update-error formulas below are specialized to the
retained \texttt{MA-ES} equations.

Having fixed the design objective and the outer optimizer, we next summarize
the complete tuning procedure before introducing its mathematical details.


\section{Mathematical formulation and algorithmic analysis}\label{sec:algorithm}

The preceding section defines the tuning problem; the next question is how the complete procedure operates. The goal here is to give a compact roadmap from problem specification to final controller validation. We organize the method into six steps, whose precise definitions and guarantees are developed in the following section.

	\noindent\textbf{Inputs:} \texttt{MPPI} and robot models; a fixed scenario law
	$\mathcal D$; a task metric and binary safety event;
	$0<\delta,\alpha_s,\alpha_v<1$; positive integers
	$G_{\max},\lambda,N,H,B,R,L_{\mathrm{look}}$; strictly increasing positive integers
	$K_1<\cdots<K_{L_{\mathrm{look}}}$; nonnegative update tolerances; and all remaining
	problem parameters satisfying the positivity and boundedness conditions
	stated below. These conditions define the admissible configuration of the
	method and must be checked by any implementation before execution.

\medskip\noindent
\textbf{Algorithm 1: Update-aware \texttt{MA-ES} tuning of \texttt{MPPI}}
\begin{enumerate}[leftmargin=1.6em,itemsep=2pt,topsep=4pt]
	\item \textbf{(S1)} Specify the fixed task metric, full-trial safety event, scenario law,
	parameter map, and common online budget $(H,B,R)$.
	\item \textbf{(S2)} Initialize the standard \texttt{MA-ES} state. At each generation, sample the
	population and keep these candidates fixed throughout that generation.
	\item \textbf{(S3)} Evaluate all candidates on a fresh shared scenario stream, extending
	its prefix only to predeclared sample sizes $K_1,\ldots,K_{L_{\mathrm{look}}}$.
	\item \textbf{(S4)} At each look, form simultaneous risk and cost-difference intervals,
	construct the certified partial order, and choose an empirical-priority
	linear extension. Compute the induced bounds on the four \texttt{MA-ES} updates.
\item \textbf{(S5)} Stop the race when all update tolerances
hold or the generation's budget is exhausted. Using the chosen order, form
the \texttt{MA-ES} recombination direction and compute the standard
evolution-path, matrix, and scale updates. When the recombination direction
is rank-certified, refine the recombination point by the nonmonotone
derivative-free line search; otherwise retain the standard recombination
point. Record the bounds and stopping status, and repeat while the search
budget permits.
	\item \textbf{(S6)} Select on separate development data. Freeze one controller and assess
	it once on $N$ untouched trials using a one-sided binomial upper bound.
\end{enumerate}

This six-step outline provides the roadmap for the detailed construction that follows. The fixed-$K$, empirical-ranking variant provides a natural matched-budget baseline for assessing the statistical evaluator, but a separate fixed-$K$ ablation is not part of the numerical study reported below. The finite-look confidence bounds used below are deliberately simple and fully distribution-free; the subsequent rank-to-update construction only requires that any substituted bounds provide the same simultaneous-coverage property. The \texttt{MPPI} horizon $H$, sample count $B$, and refinement count $R$ differ
from the \texttt{MA-ES} population size $\lambda$, outer sample counts $K_\ell$ ($\ell=1,\ldots,L_{\mathrm{look}}$),
and complete-trial horizon $T$.


\section{Mathematical formulation and algorithmic analysis}\label{sec:mathform}

The high-level algorithm identifies the six stages of the method, but its statistical and optimization guarantees depend on how those stages are defined. The goal of this section is to make each step precise and to connect trial uncertainty to the resulting optimizer update. We follow Steps (S1)--(S6) in exactly the same order as Algorithm~1.

\subsection{(S1): Define a robot rollout and its cost}

	A fair tuning procedure first requires every candidate to be judged by the
	same definition of task performance and safety. The goal of this step is
	therefore to define one complete robot trial independently of the optimizer.
	We specify the closed-loop rollout, the fixed task cost, the failure event,
	and the \texttt{MPPI} policy that generates the commands.

Let $x_k$ be the robot state, $o_k$ its observation, and $u_k$ the control
applied at time $k$. For all $k=0,\ldots,T-1$, a general closed-loop model is
\begin{equation}
\label{eq:robot}
x_0\sim\rho_0(\cdot\mid\omega),\ \
o_k=h_\omega(x_k,\eta_k),\ \
u_k=\operatorname{sat}_{\mathcal U}
       \bigl(\pi_\theta(o_{0:k})\bigr),\ \
x_{k+1} =F_\omega(x_k,u_k,\xi_k).
\end{equation}
Here, $T$ is the rollout horizon, $\rho_0$ the initial-state distribution,
$\eta_k$ observation uncertainty, $\xi_k$ motion uncertainty, and
$\mathcal U$ the admissible actuator set. A memoryless controller may use
only $o_k$ instead of the full history $o_{0:k}$. The model and scenario
distribution must represent the operating conditions relevant to the task.
For \texttt{MPPI}, the policy also maintains a warm-start command sequence and uses
a state estimate derived from $o_{0:k}$. The scenario $\omega$ includes
all environment disturbances and independent planner random seeds, so that
$C(\theta,\omega)$ also accounts for finite-sample \texttt{MPPI} variability.

	Let $\mathcal E_k(\theta,\omega)$, $k=0,\ldots,T$, denote the prescribed
	measurable event that a safety failure has occurred along the closed-loop
	trajectory up to state $x_k$, including the executed transitions leading to
	that state. Define the stopping time and complete-trial failure indicator by
	\begin{equation}
		\label{eq:failure}
			\tau(\theta,\omega)
			=
			\min\!\left(
			\{k\in\{0,\ldots,T\}:
			\mathcal E_k(\theta,\omega)\}\cup\{T\}
			\right),\ \
			I_{\mathrm{fail}}(\theta,\omega)
			=
			\mathbf 1_{\{\mathcal E_T(\theta,\omega)\}}.
	\end{equation}
	The event $\mathcal E_k$ may depend on sampled states and on executed
	transitions. For the mobile-robot instance below, it includes intersection
	of an executed motion segment with the inflated obstacle geometry. The trajectory cost is
\begin{equation}
\label{eq:rollout-cost}
C(\theta,\omega)=
\sum_{k=0}^{\tau(\theta,\omega)-1}\ell_\omega(x_k,u_k)
+\Phi_\omega(x_{\tau(\theta,\omega)})
+L_{\mathrm{fail}}I_{\mathrm{fail}}(\theta,\omega).
\end{equation}
Here, $L_{\mathrm{fail}}\ge0$ is a fixed failure penalty.
The running cost $\ell_\omega$ describes progress and control expenditure;
the terminal cost $\Phi_\omega$ describes the final task outcome. The
failure indicator is retained separately even if a failure penalty is
included in the cost: a penalty alone cannot establish a safety constraint. In the reported implementation, a trial terminates at the first
detected collision or at the prescribed time horizon. Entering
the goal region does not terminate the trial; goal completion is
evaluated at the final simulated state.

\subsubsection{\texttt{MPPI} as the online planning policy}

At time $k$, let $\widehat x_k$ be the estimated state and
$U=(\bar u_0,\ldots,\bar u_{H-1})$ the nominal sequence of
\emph{unconstrained} command variables, with $\bar u_h\in\mathbb R^d$,
where $d$ is the control dimension.
The prediction model $\widehat F_k$ uses only information available at
time $k$ and need not equal the trial dynamics $F_\omega$.
For $b=1,\ldots,B$, independently sample
\begin{equation}
\label{eq:mppi-samples}
\epsilon_h^{(b)}\sim\mathcal N(0,\Sigma_\theta),\qquad
v_h^{(b)}=\bar u_h+\epsilon_h^{(b)},\qquad h=0,\ldots,H-1,
\end{equation}
where $\Sigma_\theta\succ0$. Let
$V^{(b)}=(v_0^{(b)},\ldots,v_{H-1}^{(b)})$. Evaluate bounded commands and predicted
states using
\begin{equation}
\label{eq:mppi-rollout}
a_h^{(b)}=\operatorname{sat}_{\mathcal U}(v_h^{(b)}),\qquad
\widehat x_0^{(b)}=\widehat x_k,\qquad
\widehat x_{h+1}^{(b)}=\widehat F_k(\widehat x_h^{(b)},a_h^{(b)}).
\end{equation}
The planning score before importance correction is
\begin{equation}
\label{eq:mppi-cost}
\mathcal S_{\theta,k}(V^{(b)})=
\phi_{\theta,k}(\widehat x_H^{(b)})+
\sum_{h=0}^{H-1}
q_{\theta,k,h}(\widehat x_h^{(b)},a_h^{(b)},a_{h-1}^{(b)}),
\qquad a_{-1}^{(b)}=u_{k-1}.
\end{equation}
Unlike the outer task metric~\eqref{eq:rollout-cost}, the planner score
may contain tunable weights. The outer functions $\ell_\omega$,
$\Phi_\omega$, penalty $L_{\mathrm{fail}}$, and failure specification
are fixed for all candidates. Otherwise a candidate could appear better
simply by lowering the weights used to evaluate itself.

For clarity, use a zero-mean reference density
$p_\theta(V)=\prod_h\mathcal N(v_h;0,\Sigma_\theta)$ and a proposal
$q_{U,\theta}(V)=\prod_h\mathcal N(v_h;\bar u_h,\Sigma_\theta)$.
For $\beta>0$, let $Z_{\theta,k}$ denote the normalizing constant and define the information-theoretic target density
\begin{equation}
\label{eq:mppi-target}
q^*_{\theta,k}(V)=Z_{\theta,k}^{-1}
\exp[-\mathcal S_{\theta,k}(V)/\beta]p_\theta(V).
\end{equation}
Assume a measurable score, $0<Z_{\theta,k}<\infty$, finite target
second moments, and finite differential entropy. These conditions hold
for the finite nonnegative navigation scores specified below. Projecting this density onto the Gaussian family with
fixed covariance by minimizing
$D_{\mathrm{KL}}(q^*_{\theta,k}\Vert q_{U,\theta})$ gives
$\bar u_h^*=\mathbb E_{q^*_{\theta,k}}[v_h]$.
This is a distribution-projection statement, not a claim that a
finite-sample command sequence globally solves the constrained robot
problem~\cite{williams2018}.

Since $V=U+\epsilon$, direct expansion of the two Gaussian densities gives
\begin{equation}
\label{eq:mppi-ratio}
\log\frac{p_\theta(V)}{q_{U,\theta}(V)}
=-\sum_{h=0}^{H-1}\bar u_h^{\mathsf T}\Sigma_\theta^{-1}\epsilon_h
-\frac12\sum_{h=0}^{H-1}\bar u_h^{\mathsf T}\Sigma_\theta^{-1}\bar u_h.
\end{equation}
The last term is common to all samples and cancels in normalized weights.
The resulting \texttt{MPPI} update is
\begin{align}
\widetilde{\mathcal S}_b
&=\mathcal S_{\theta,k}(V^{(b)})+
\beta\sum_{h=0}^{H-1}
\bar u_h^{\mathsf T}\Sigma_\theta^{-1}\epsilon_h^{(b)},
\label{eq:mppi-correction}\\
\rho&=\min_{1\leq b\leq B}\widetilde{\mathcal S}_b,\qquad
W_b=\frac{\exp[-(\widetilde{\mathcal S}_b-\rho)/\beta]}
{\sum_{j=1}^B\exp[-(\widetilde{\mathcal S}_j-\rho)/\beta]},
\label{eq:mppi-weights}\\
\bar u_h^+&=\bar u_h+\sum_{b=1}^B W_b\epsilon_h^{(b)},
\qquad h=0,\ldots,H-1.
\label{eq:mppi-update}
\end{align}
This self-normalized importance-sampling estimator is generally biased
at finite $B$, although consistent under the usual integrability
conditions. Subtracting $\rho$ stabilizes its numerical evaluation.
Omitting the density-ratio correction would change the target of the
specified update. The same $\Sigma_\theta$ is used in the reference and
proposal; different covariances would require their full density ratio.

All densities and updates above refer to the raw variables $v_h$ and
$\epsilon_h$, while saturation is inside the rollout model. Replacing
$\epsilon_h$ in~\eqref{eq:mppi-correction}--\eqref{eq:mppi-update}
by clipped perturbations would not implement this Gaussian derivation.
After a fixed number $R$ of sampling/update passes, apply
\begin{equation}
\label{eq:mppi-apply}
u_k=\operatorname{sat}_{\mathcal U}(\bar u_0^+),\qquad
U_{k+1}^{\mathrm{init}}=
(\bar u_1^+,\ldots,\bar u_{H-1}^+,0).
\end{equation}
Initialize $U_0^{\mathrm{init}}=0$ and redraw perturbations in every pass.
This recursion, the state estimator, and the warm start specify
$\pi_\theta^{\mathrm{\texttt{MPPI}}}$ in~\eqref{eq:robot}.
The random perturbations are computational exploration; they need not
equal physical actuation noise. A deterministic prediction model is used
here, while model error and disturbances remain in the outer scenarios.

\subsubsection{Instantiation for mobile-robot motion planning}

A concrete instance is a planar robot with pose
$z=(p_x,p_y,\psi)$, position $p=(p_x,p_y)$, and commanded linear and
angular velocities $u=(v,r)$. A kinematic prediction model is
\begin{equation}
\label{eq:unicycle}
\widehat F(z,u)=
\begin{bmatrix}
p_x+\Delta t\,v\cos\psi\\
p_y+\Delta t\,v\sin\psi\\
\psi+\Delta t\,r
\end{bmatrix},\qquad
\mathcal U=[0,v_{\max}]\times[-r_{\max},r_{\max}].
\end{equation}
Here, $\Delta t>0$ is the prediction time step, while
$v_{\max}>0$ and $r_{\max}>0$ are the prescribed linear- and
angular-velocity limits.
Section~\ref{sec:simulation} instantiates this model with the command
limits and conservative footprint of a \texttt{TurtleBot~4}. The mathematical
formulation itself remains platform-independent. When inertia, slip, or
actuator dynamics matter, $\widehat F$ is replaced by an appropriate
identified model without changing either optimization recursion.
For costs involving the previous input, use the augmented state
$x_k=(z_k,u_{k-1})$, with $u_{-1}=0$, so that~\eqref{eq:robot}
and~\eqref{eq:rollout-cost} retain their stated form. The augmented
transition is $\widehat F_x((z,u^-),u)=(\widehat F(z,u),u)$. Here, $u^-$ denotes the previous executed command and $a^-$ the
previous bounded planning command.

Let $p_g$ be a goal, $\Gamma$ a fixed reference route, and $L>0$ a
fixed distance scale. For disk obstacles with centers $c_j$ and radii
$R_j$, and a disk robot of radius $R_b$, define
\begin{equation}
\label{eq:features}
e_g(p)={\|p-p_g\|}/{L},\qquad
e_\Gamma(p)={\operatorname{dist}(p,\Gamma)}/{L},\qquad
d(p)=\min_j\{\|p-c_j\|-R_j-R_b\}.
\end{equation}
Set $\mathcal X_{\mathrm{safe}}=\{x:d(p)>0\}$ and let
$b(p)=[\max\{0,1-d(p)/d_s\}]^2$, where $d_s>0$ is a fixed clearance
margin. With $D_u=\operatorname{diag}(v_{\max},r_{\max})$, choose
\begin{align}
q_\theta(x,a,a^-)
&=e_g(p)^2+q_p e_\Gamma(p)^2+q_o b(p)
 +q_\Delta\|D_u^{-1}(a-a^-)\|^2
 +L_{\mathrm{col}}\mathbf 1_{\{d(p)\leq0\}},
\label{eq:nav-stage}\\
\phi_\theta(x)
&=\kappa_g e_g(p)^2+
 L_{\mathrm{col}}\mathbf 1_{\{d(p)\leq0\}}.
\label{eq:nav-terminal}
\end{align}
Here, $L_{\mathrm{col}}>0$ and $\kappa_g>0$ are fixed. The fixed unit
coefficient on $e_g^2$ anchors the scale of the running score.
The six tuned quantities and command covariance are
\begin{equation}
\label{eq:theta-mppi}
\theta=(q_p,q_o,q_\Delta,\beta,s_v,s_r)\in\Theta\subset(0,\infty)^6,
\qquad \Sigma_\theta=\operatorname{diag}(s_v^2,s_r^2).
\end{equation}
Here, $s_v$ denotes the translational-speed sampling scale.
For the two unicycle models, $s_r$ denotes the yaw-rate
sampling scale. For the \texttt{Ackermann} model, the second
control component is the steering angle; consequently,
$s_r$ denotes its sampling standard deviation in radians,
rather than a yaw-rate sampling scale. The quantities $H,B,R,\Delta t$, actuator limits, footprint, clearance
margin, and collision definition are fixed design specifications.
In particular, \texttt{MA-ES} does not change physical safety margins to improve
its score. Its matrix $M_g\in\mathbb R^{6\times6}$ adapts the search
over planner parameters; it is distinct from the \texttt{MPPI} covariance
$\Sigma_\theta\in\mathbb R^{2\times2}$.

For this navigation instance, one explicit fixed outer metric is
\begin{align}
\ell_\omega(x,u)
&=a_g e_g(p)^2+a_u\|D_u^{-1}u\|^2
 +a_\Delta\|D_u^{-1}(u-u^-)\|^2,
\label{eq:nav-outer-stage}\\
\Phi_\omega(x)
&=a_T e_g(p)^2+
 a_{\mathrm{miss}}\mathbf1_{\{\|p-p_g\|>r_g\}},
\label{eq:nav-outer-terminal}
\end{align}
where all $a$ coefficients are nonnegative and fixed before tuning,
and goal tolerance $r_g$ and $L_{\mathrm{fail}}$ are shared by every
candidate. Obstacle geometry, initial pose, and motion/sensing errors
vary with $\omega$. Estimated geometry is used for online planning;
the scenario geometry is used to evaluate collisions.
This completes the mapping from \texttt{MA-ES} candidates to robot commands
and to the scalar cost and failure indicator in~\eqref{eq:problem}.

Obstacle penalties in~\eqref{eq:nav-stage}--\eqref{eq:nav-terminal}
are soft preferences and do not guarantee collision avoidance.
In the reported implementation, obstacle costs inside the
\texttt{MPPI} prediction model are evaluated at the sampled
predicted states. The executed closed-loop trajectory is
evaluated using swept-segment collision checks: each straight
segment connecting consecutive simulated positions is tested
against the obstacle geometry after footprint inflation.
This detects intersample collisions along the piecewise-linear
trajectory represented by the discrete simulator, but does not
establish collision avoidance for continuous-time physical motion.

\subsection{(S2): Initialize the \texttt{MA-ES} search}

	Once the controller parameters are fixed, the optimizer needs a search
	representation that respects their admissible ranges. The goal is to allow
	\texttt{MA-ES} to explore freely while ensuring that every evaluated
	controller remains admissible. We therefore search in a latent space and map
	each sampled point to bounded physical controller parameters.

\texttt{MA-ES} searches over an unconstrained vector $y\in\mathbb R^n$. A mapping
$\mathcal T:\mathbb R^n\rightarrow\Theta$ converts it to legal controller
parameters, $\theta=\mathcal T(y)$. For the positive \texttt{MPPI} parameters,
fix bounds $0<\theta_j^-<\theta_j^+<\infty$ and use
\begin{equation}
\label{eq:parameter-map}
\theta_j=\theta_j^-\exp\!\left[
\log\!\left({\theta_j^+}/{\theta_j^-}\right)
\frac{1}{1+e^{-y_j}}\right],\qquad j=1,\ldots,n.
\end{equation}
This parameterizes the open bounded search domain
$\Theta=\prod_j(\theta_j^-,\theta_j^+)$ and ensures strictly positive
temperature and sampling scales. It preserves all latent \texttt{MA-ES} updates;
the chosen nonlinear parameterization does affect the search geometry.

	Equation~\eqref{eq:parameter-map} defines an open mathematical
	parameterization. In finite-precision implementations, its numerical
	evaluation should use a stable sigmoid and retain the result in the
	representable strict interior of
	$(\theta_j^-,\theta_j^+)$; nonfinite latent vectors are inadmissible.
	This numerical safeguard does not change the mathematical parameter map.

Choose $m_0\in\mathbb R^n$, $\sigma_0>0$, $M_0=I_n$, and $s_0=0$. The variables are respectively the mean, global scale, transformation
matrix, and evolution path, and $I_n$ denotes the $n\times n$ identity matrix. A standard population choice is
\begin{equation}
\label{eq:pop}
\lambda=\lfloor4+3\ln n\rfloor,
\qquad \mu=\lfloor\lambda/2\rfloor.
\end{equation}
For the best $\mu$ ranks, define positive weights and an effective
selection size:
\begin{equation}
\label{eq:weights}
\widetilde w_j=\ln\!\left((\lambda+1)/{2}\right)-\ln j,
\quad
w_j=\frac{\widetilde w_j}{\sum_{r=1}^{\mu}\widetilde w_r},
\quad
\mu_{\mathrm{eff}}=\left(\sum_{j=1}^{\mu}w_j^2\right)^{-1},
\quad j=1,\ldots,\mu.
\end{equation}
For the \texttt{MA-ES} updates used below, define
\begin{equation}
\label{eq:norm}
\chi_n
:=\mathbb E_{Z\sim\mathcal N(0,I_n)}\|Z\|_2
=\sqrt{2}\,\frac{\Gamma((n+1)/2)}{\Gamma(n/2)}.
\end{equation}
The remaining \texttt{MA-ES} parameters are
\begin{equation}
\label{eq:rates}
\begin{aligned}
c_s&=\frac{\mu_{\mathrm{eff}}+2}{n+\mu_{\mathrm{eff}}+5}\in (0,2), \ \ c_1 =\frac{2}{(n+1.3)^2+\mu_{\mathrm{eff}}},\\
c_\mu&=\min\!\left\{1-c_1,
\frac{2(\mu_{\mathrm{eff}}-2+1/\mu_{\mathrm{eff}})}
{(n+2)^2+\mu_{\mathrm{eff}}}\right\},\\
d_\sigma&=1+c_s+2\max\!\left\{0,
\sqrt{(\mu_{\mathrm{eff}}-1)/(n+1)}-1\right\}.
\end{aligned}
\end{equation}
Since $n\ge1$ and $\mu_{\mathrm{eff}}\ge1$, the square root in
$d_\sigma$ is real. The particular choice of $c_s$ satisfies $0<c_s<1$,
so $\sqrt{\mu_{\mathrm{eff}}c_s(2-c_s)}$ is real. Moreover,
$c_1>0$, $c_\mu\ge0$, $d_\sigma>0$, and $\chi_n>0$.

\subsection{(S3): Sample candidates and perform full evaluations}

	Candidate controllers must be compared under uncertainty without giving some
	candidates systematically easier trials than others. The goal of this step is
	to obtain fair within-generation comparisons while exposing later generations
	to fresh conditions. We therefore use shared trial scenarios within each
	generation and refresh them across generations.

For each $i\in\{1,\ldots,\lambda\}$ in generation $g$, draw
\begin{equation}
\label{eq:sampling}
z_i^{(g)}\sim\mathcal N(0,I_n),\qquad
y_i^{(g)}=m_g+\sigma_gM_gz_i^{(g)},\qquad
\theta_i^{(g)}=\mathcal T\bigl(y_i^{(g)}\bigr).
\end{equation}
Thus the latent sampling distribution has mean $m_g$ and covariance
$\sigma_g^2M_gM_g^{\mathsf T}$. This covariance is implied by the
transformation; \texttt{MA-ES} does not explicitly maintain a separate covariance
matrix.

At a sampling look with $K=K_\ell$, use a common scenario prefix
$\omega_1^{(g)},\ldots,\omega_K^{(g)}$ for the generation. Run the complete
closed-loop system~\eqref{eq:robot} for every candidate and every scenario,
and compute
\begin{equation}
\label{eq:estimates}
\widehat J_g(\theta_i^{(g)})
=\frac1K\sum_{r=1}^K C(\theta_i^{(g)},\omega_r^{(g)}),\ \
\widehat p_g(\theta_i^{(g)})
=\frac1K\sum_{r=1}^K
  I_{\mathrm{fail}}(\theta_i^{(g)},\omega_r^{(g)}).
\end{equation}
Common scenarios make within-generation comparisons less dependent on
which candidate happened to encounter an easier environment. Refreshing
scenarios in later generations reduces dependence on one fixed training set.
For \texttt{MPPI}, share the standard-normal draws $e_{k,a,h}^{(b)}$ within a
generation and construct candidate-specific perturbations
$\epsilon_{k,a,h}^{(b)}=\Sigma_\theta^{1/2}e_{k,a,h}^{(b)}$.
Here, $a=1,\ldots,R$ indexes refinement passes; the scenario also indexes
the trial. Sharing the realized perturbations irrespective of $\Sigma_\theta$
would implement the wrong proposal for some candidates.
Use independent scenario seeds across trials, reset the state estimator
and warm start at the beginning of each trial, and refresh both environment
and planner seeds across generations. Shared random numbers facilitate
paired comparisons; they do not universally guarantee variance reduction.

The prefix is extended, not redrawn, between looks of the same generation.
Fresh scenario streams are independent across generations, conditional on
all previously selected search states. Candidates cannot be modified in
response to partial evaluations within a race.
\subsection{(S4): Risk comparisons and update-aware statistical racing}
\label{sec:racing}

	Completely resolving every population ranking can require many expensive
	robot trials even when the remaining ambiguity would barely affect the
	optimizer. The goal of this step is to stop evaluation when ranking
	uncertainty is sufficiently small for the \texttt{MA-ES} update. We achieve
	this by converting statistically supported comparisons into bounds on the
	possible optimizer update.

\subsubsection{Empirical reference ordering and assumptions}

For the remainder of this subsection, fix the current generation and
look and suppress the generation superscript when no ambiguity arises:
$\theta_i:=\theta_i^{(g)}$, $z_i:=z_i^{(g)}$, and
$\omega_r:=\omega_r^{(g)}$. We also write
$\widehat J_i:=\widehat J_g(\theta_i^{(g)})$ and
$\widehat p_i:=\widehat p_g(\theta_i^{(g)})$. The empirical-ranking ablation orders the population by
\begin{equation}
\label{eq:rank}
\left(\max\{0,\widehat p_i-\delta\},\widehat J_i,i\right),
\end{equation}
where $i$ is a fixed tie-breaking index. This feasibility-first construction
is consistent with the established separation of objective and violation
in evolutionary constraint handling~\cite{deb2000}. The target ordering for
the analysis replaces the estimates by $p_i=p_{\mathrm{fail}}(\theta_i)$
and $J_i=J(\theta_i)$. It is an ordering of the \emph{realized population},
not an assumption that a globally optimal controller lies in it. The
feasibility-first lexicographic rule is a constraint-handling search
heuristic: it prioritizes lower violation when no feasible candidate is
available and lower cost on the feasible plateau, but by itself it is not a
convergence proof for the chance-constrained problem~\eqref{eq:problem}.

We impose three additional conditions for statistical racing. Conditional
on all previous generations, the candidates are fixed and scenarios in
the current stream are independent draws from the same $\mathcal D$.
The generation cap and finite sample-size grid are specified before the
search. Finally, we assume that fixed, known bounds satisfy
\begin{equation}
	\label{eq:bounded-cost}
	C_{\min}\le C(\theta,\omega)\le C_{\max}
	\quad\text{a.s. for every }\theta\in\Theta,
	\qquad W=C_{\max}-C_{\min}>0,
\end{equation}
where ``a.s.'' means ``almost surely'' with respect to the
randomness in $\omega$. The last condition is not implied by a
finite empirical data set or by Gaussian \texttt{MPPI} sampling. It must follow from the declared task and scenario
support, for example bounded initial conditions and physical disturbances,
bounded inputs, a finite trial horizon, and bounded cost features on the
resulting reachable set. The Gaussian \emph{internal} command samples can
remain unbounded because executed commands are saturated. With unbounded
external cost or disturbance support, the following cost bounds must be
replaced by a justified alternative; clipping measured outcomes after the
fact would change the objective and is not part of this method.

\subsubsection{Simultaneous risk and paired-cost intervals}
For a predeclared search error budget $0<\alpha_s<1$, set
\begin{equation}
\label{eq:error-allocation}
Q=G_{\max}L_{\mathrm{look}}\left[\lambda+\binom{\lambda}{2}\right],
\qquad \varepsilon=\alpha_s/Q.
\end{equation}
At a look with $K$ trials and $f_i$ observed failures, use the two-sided
Clopper--Pearson interval $[p_i^-,p_i^+]$ with noncoverage at most
$\varepsilon$~\cite{clopper1934}:
\begin{equation}
\label{eq:cp-search}
\begin{aligned}
p_i^-&=\begin{cases}0,&f_i=0,\\
\operatorname{B}^{-1}(\varepsilon/2;f_i,K-f_i+1),&f_i>0,\end{cases}\\
p_i^+&=\begin{cases}1,&f_i=K,\\
\operatorname{B}^{-1}(1-\varepsilon/2;f_i+1,K-f_i),&f_i<K.
\end{cases}
\end{aligned}
\end{equation}
Here, $\operatorname{B}^{-1}(q;a,b)$ is a beta-distribution quantile.
Define $v_i^-=(p_i^--\delta)_+$ and $v_i^+=(p_i^+-\delta)_+$,
where $(t)_+=\max(0,t)$.
For each unordered candidate pair use the paired observations
$D_{ij,r}=C(\theta_i,\omega_r)-C(\theta_j,\omega_r)\in[-W,W]$.
Hoeffding's inequality~\cite{hoeffding1963} gives
\begin{equation}
\label{eq:paired-cost}
\widehat d_{ij}=\frac1K\sum_{r=1}^K D_{ij,r},\quad
r_K=W\sqrt{{2\log(2/\varepsilon)}/{K}},\quad
J_i-J_j\in[\widehat d_{ij}-r_K,\widehat d_{ij}+r_K].
\end{equation}
Reverse-direction intervals are obtained by negation, not treated as extra
tests. Dependence between candidates sharing a scenario is permitted;
independence is required across trials. Pairing can change estimator
variance but is not asserted to improve this distribution-free radius.
The Hoeffding radius is intentionally assumption-light and can be very
conservative when the known global cost range $W$ is much larger than the
observed paired variability. Proposition~\ref{prop:update} does not depend
on Hoeffding specifically: tighter finite-look or sequential intervals may
be substituted if their assumptions are justified and their coverage is
controlled simultaneously over the comparisons actually used.

Predeclaration of the look grid alone does not ensure that useful comparisons can be certified. Before an experiment, the largest planned look should therefore be checked analytically for statistical resolvability: the risk intervals must be capable of separating relevant
 failure probabilities or certifying the safe plateau, and the cost
	interval radius should be small enough to resolve practically meaningful
	cost differences. The quick-budget study below deliberately remains a
	proof-of-execution study and does not satisfy these stronger
	resolvability requirements.

\begin{proposition}[Simultaneous coverage]
\label{prop:coverage}
Under the stated assumptions, all risk and paired-cost intervals used at
all visited looks and generations simultaneously cover their respective
population quantities with probability at least $1-\alpha_s$.
\end{proposition}
\begin{proof}
For a fixed generation, condition on its incoming history and sampled
population. Each risk count is binomial and each paired-cost sequence is
i.i.d. and bounded. Each interval therefore has conditional noncoverage at
most $\varepsilon$. A union bound over the predeclared looks, candidates,
and unordered pairs, followed by averaging over the incoming history and
summing over at most $G_{\max}$ generations, gives $Q\varepsilon=\alpha_s$.
Stopping a race at a visited look selects a subset of these simultaneous
statements; it does not introduce an unadjusted optional-stopping claim.
\end{proof}

\subsubsection{From safe comparisons to admissible ranks}
Add a directed edge $i\prec j$ when either of the following holds:
\begin{equation}
\label{eq:dominance}
\begin{aligned}
&v_i^+<v_j^-, \quad\text{or}\quad p_i^+\le\delta,\quad p_j^+\le\delta,\quad
  \widehat d_{ij}+r_K<0.
\end{aligned}
\end{equation}
The first rule separates true constraint violations. The second compares
cost only when both candidates are certified to lie on the zero-violation
plateau. Overlapping violation intervals, or an empirical tie between
infeasible candidates, do not justify a cost-only certified comparison.
Retain edges from previous looks and take their transitive closure.
On the coverage event, every edge agrees with the true lexicographic
ordering, so the graph is acyclic and that true ordering is a linear
extension of it.

Choose a deterministic extension $\widehat\pi$ by repeatedly selecting,
among currently predecessor-free candidates, the smallest empirical key
in~\eqref{eq:rank}. This respects certified edges even when the raw
empirical ordering does not. Let $P_i$ and $S_i^{\mathrm{succ}}$ be the
strict predecessor and successor sets in the transitive closure. Every
admissible ordering places candidate $i$ at a rank in
\begin{equation}
\label{eq:rank-range}
\underline r_i=1+|P_i|,\qquad
\overline r_i=\lambda-|S_i^{\mathrm{succ}}|.
\end{equation}
The bounds in~\eqref{eq:rank-range} are outer bounds obtained from
predecessor and successor counts. Not every integer rank inside an
individual interval need occur in a common linear extension, so treating
the candidates independently through these intervals is conservative but
safe for the update-discrepancy bounds below. If numerical computation
produces a cycle, the look is not certified; record the inconsistency and
use the empirical fallback rather than claiming a confidence-supported
order.

\subsubsection{From rank ambiguity to \texttt{MA-ES} update bounds}
Extend the positive, nonincreasing recombination weights from
\eqref{eq:weights} by $w_j=0$ for $j>\mu$. Let
$\widehat a_i=w_{\operatorname{rank}_{\widehat\pi}(i)}$ and define
\begin{equation}
\label{eq:weight-radius}
a_i^-=w_{\overline r_i},\ \ a_i^+=w_{\underline r_i},\ \
b_i=\max\{\widehat a_i-a_i^-,\ a_i^+-\widehat a_i\},\ \
E_z=\sum_{i=1}^{\lambda}b_i\|z_i\|_2,\ \
E_S=\sum_{i=1}^{\lambda}b_i\|z_i\|_2^2.
\end{equation}
The current extension yields
\begin{equation}
\label{eq:recombination}
\widehat{\bar z}_g=\sum_i\widehat a_i z_i^{(g)},\qquad
\widehat S_g=\sum_i\widehat a_i z_i^{(g)}z_i^{(g)\mathsf T}.
\end{equation}
Hereafter the unhatted $\bar z_g,S_g$ in the standard recursion denote
these chosen quantities. A star denotes a counterfactual update using the
\emph{true ordering of this same population and the same incoming state}.
Set
\begin{equation}
\label{eq:path-radius}
a_s=\sqrt{\mu_{\mathrm{eff}}c_s(2-c_s)},\qquad h=a_s E_z.
\end{equation}
We call the recombination direction \emph{rank-certified} when the
	sufficient certificate $E_z=0$ holds. On the simultaneous coverage event,
	this guarantees that the chosen-order and true-order recombination
	directions coincide.
\begin{proposition}[Population-conditional update discrepancy]
\label{prop:update}
On the simultaneous coverage event, for every visited acyclic look,
\begin{align}
\|\widehat{\bar z}_g-\bar z_g^*\|_2&\le E_z, \ \ 
\|\widehat S_g-S_g^*\|_2\le E_S,\label{eq:stat-bounds}\\
\|\widehat m_{g+1}-m_{g+1}^*\|_2&\le\sigma_g\|M_g\|_2 E_z,
 \ \ \|\widehat s_{g+1}-s_{g+1}^*\|_2\le h,\label{eq:mean-path-bounds}\\
\|\widehat M_{g+1}-M_{g+1}^*\|_2
&\le\|M_g\|_2\left[\frac{c_1}{2}(2\|\widehat s_{g+1}\|_2+h)h
+\frac{c_\mu}{2}E_S\right],\label{eq:matrix-bound}\\
\left|\log\frac{\widehat\sigma_{g+1}}{\sigma_{g+1}^*}\right|
&\le \frac{c_s}{d_\sigma\chi_n}h.\label{eq:scale-bound}
\end{align}
Vector norms are Euclidean and matrix norms are spectral.
\end{proposition}
\begin{proof}
The true ordering is an admissible extension. Monotonicity of the extended
weights gives $a_i^*\in[a_i^-,a_i^+]$, hence
$|\widehat a_i-a_i^*|\le b_i$. The triangle inequality and
$\|z_i z_i^{\mathsf T}\|_2=\|z_i\|_2^2$ prove
\eqref{eq:stat-bounds}. Substitution into the mean and path recursions
proves~\eqref{eq:mean-path-bounds}. Writing
$e=\widehat s_{g+1}-s_{g+1}^*$ gives
\[
\|\widehat s_{g+1}\widehat s_{g+1}^{\mathsf T}
-s_{g+1}^*s_{g+1}^{*\mathsf T}\|_2
\le 2\|\widehat s_{g+1}\|_2\|e\|_2+\|e\|_2^2.
\]
The matrix recursion and submultiplicativity imply
\eqref{eq:matrix-bound}. Taking the difference of the two log-scale
updates and applying the reverse triangle inequality gives
\eqref{eq:scale-bound}. 
\end{proof}
Before the batch study, unit tests checked the latent-to-physical parameter transform, finite \texttt{MA-ES} updates, the exact zero-failure Clopper--Pearson identity, swept-disk collision geometry, \texttt{MPPI} reproducibility and actuator bounds, deterministic complete-trial reproducibility, and the analytic outer-cost bound. These are software-consistency checks; the statistical guarantees of the method follow from the analytical results above rather than from the unit tests.

\subsubsection{Update-aware stopping and budget exhaustion}
At each look, evaluate the right-hand sides of
\eqref{eq:mean-path-bounds}--\eqref{eq:scale-bound}. Stop when they are no
larger than the predeclared tolerances
$(\eta_m,\eta_s,\eta_M,\eta_{\log\sigma})$, respectively, or when $K_{L_{\mathrm{look}}}$
is reached. The tolerances concern latent search-state changes, not robot
tracking errors. They are therefore parameterization- and scale-dependent,
not invariant physical error tolerances. They must be fixed before the run
and reported; zero tolerances recover exact-update certification whenever
the bounds vanish.
An unresolved ordering among candidates certainly outside the elite set
has zero weight and need not be resolved. An unresolved ordering among
elite candidates generally does matter because the weights are nonuniform.
Small update sensitivity can also permit termination without an exact
elite ranking.

At the cap, apply the chosen extension but explicitly label the generation
\emph{tolerance-unresolved} if any bound exceeds its tolerance. The error
bounds remain statements on the coverage event; budget exhaustion does not
turn them into small-error guarantees. Exact risk ties, candidates near
$\delta$, broad cost bounds, and conservative multiple-testing correction
can prevent useful early stopping. No sample-efficiency improvement is
asserted without experiments. The guarantee compares one update at the
realized incoming state, not the entire search trajectory with a noiseless
oracle trajectory, and it is not a convergence or stability theorem.

\subsection{(S5): Update mean, path, matrix, scale, and recombination point}\label{sec:recombination-linesearch}

	After the race stops, the chosen population ordering must be converted into
	the next \texttt{MA-ES} search state. The goal is to preserve the standard
	adaptation mechanism while allowing the recombination point to be improved
	when its direction is rank-certified. We therefore apply the standard state
	updates and conditionally refine only the recombination point by a
	nonmonotone derivative-free line search.

\texttt{MA-ES} performs the following updates, with all right-hand-side generation
$g$ quantities taken before they are overwritten:
\begin{align}
	m_{g+1}
	&=m_g+\sigma_gM_g\bar z_g,
	\label{eq:mean}\\
	s_{g+1}
	&=(1-c_s)s_g+
	\sqrt{\mu_{\mathrm{eff}}c_s(2-c_s)}\,\bar z_g,
	\label{eq:path}\\
	M_{g+1}
	&=M_g\left[
	I_n+\frac{c_1}{2}
	\bigl(s_{g+1}s_{g+1}^{\mathsf T}-I_n\bigr)
	+\frac{c_\mu}{2}(S_g-I_n)
	\right],
	\label{eq:matrix}\\
	\sigma_{g+1}
	&=\sigma_g\exp\!\left[
	\frac{c_s}{d_\sigma}
	\left(\frac{\|s_{g+1}\|}{\chi_n}-1\right)
	\right].
	\label{eq:scale}
\end{align}
The coefficients $\mu_{\mathrm{eff}}$, $\chi_n$, $c_s$, $c_1$,
$c_\mu$, and $d_\sigma$ are defined in
\eqref{eq:weights}--\eqref{eq:rates}. The mean update determines the next
sampling center, while the path, transformation matrix, and scale adapt the
subsequent search distribution. These are the standard \texttt{MA-ES} state
updates used here~\cite{beyer2017}, before the conditional
recombination-point refinement described below.

	\paragraph{Nonmonotone derivative-free line search for the recombination point.}

	The standard \texttt{MA-ES} mean update in~\eqref{eq:mean} uses the weighted
	recombination direction once, with the unit scalar multiplier. In a conditional refinement, we use a derivative-free line search to determine
	whether extrapolation along this direction produces a better recombination
	point. The construction is adapted from the recombination line-search idea
	of Kimiaei and Neumaier~\cite{kimiaei2024madfo}, but we remove the additional
	randomization used in its nonmonotone reference term. The recombination
	direction itself remains stochastic because it is generated from the random
	mutation and selection process of \texttt{MA-ES}. This modification is
	natural here because the outer problem is already stochastic through
	complete robot trials and the main analysis of this paper is concerned
	with uncertainty in the \texttt{MA-ES} update.
	
	Define the generation-$g$ recombination direction in latent coordinates by
	\begin{equation}
		\label{eq:ls-recombination-direction}
		d_g^{\mathrm{rec}}:=M_g\bar z_g,
		\qquad
		m_{g+1}^{\mathrm{MA}}:=m_g+\sigma_g d_g^{\mathrm{rec}}.
	\end{equation}
	Thus $m_{g+1}^{\mathrm{MA}}$ is exactly the mean produced by the standard
	update~\eqref{eq:mean}. The line search does not change $\bar z_g$, $S_g$,
	or any recombination weight. It only tests scalar multiples of the already
	computed direction. For a scalar $\alpha>0$, define
	\begin{equation}
		\label{eq:ls-trial-point}
		y_g(\alpha):=m_g+\alpha\sigma_g d_g^{\mathrm{rec}},
		\qquad
		\theta_g(\alpha):=\mathcal T\bigl(y_g(\alpha)\bigr).
	\end{equation}
	The standard \texttt{MA-ES} recombination point corresponds to $\alpha=1$.
	No derivative of the complete-trial cost, the safety probability, the
	parameter map $\mathcal T$, or the \texttt{MPPI} closed-loop simulation is
	used.
	
\paragraph{Deterministic nonmonotone reference.}
Let $K_g$ denote the sample size reached at the final racing look of
generation $g$. All population candidates have already been evaluated on the
same prefix of $K_g$ complete scenarios. Let
\begin{equation}
	\label{eq:ls-population-cost-reference}
	\mathcal F_g:=
	\left\{
	\widehat J_g\bigl(\theta_i^{(g)}\bigr):i=1,\ldots,\lambda
	\right\},
\end{equation}
Fix an integer nonmonotone-memory length $m_{\mathrm{nm}}\geq1$ before the
run and maintain the deterministic FIFO collection
\begin{equation}
	\label{eq:ls-memory}
	\mathcal H_g^{(m_{\mathrm{nm}})}
	:=
	\bigcup_{t=\max\{0,g-m_{\mathrm{nm}}+1\}}^{g}\mathcal F_t.
\end{equation}
Thus the memory contains at most $m_{\mathrm{nm}}\lambda$ population-cost
statistics and uses no random replacement. Define
\begin{equation}
	\label{eq:ls-min-med-max}
	J_g^{\min}:=\min\mathcal H_g^{(m_{\mathrm{nm}})},
	\qquad
	J_g^{\mathrm{med}}:=\operatorname{median}\mathcal H_g^{(m_{\mathrm{nm}})},
	\qquad
	J_g^{\max}:=\max\mathcal H_g^{(m_{\mathrm{nm}})}.
\end{equation}
The choice $m_{\mathrm{nm}}=1$ recovers the current-population reference
$\mathcal H_g^{(1)}=\mathcal F_g$. Fix once, before the run, a deterministic
parameter $\eta_{\mathrm{nm}}\in(1/2,1)$. After evaluating the nominal recombination
point $\theta_g(1)$ on the same stored generation-$g$ scenario prefix, set
\begin{equation}
	\label{eq:deterministic-nonmonotone-reference}
	J_g^{\mathrm{nm}}:=
	\begin{cases}
		\eta_{\mathrm{nm}}J_g^{\mathrm{med}}
		+(1-\eta_{\mathrm{nm}})J_g^{\max},
		& \widehat J_g^{\mathrm{ls}}(1)\ge J_g^{\max},\\[1mm]
		(1-\eta_{\mathrm{nm}})J_g^{\mathrm{med}}
		+\eta_{\mathrm{nm}}J_g^{\min},
		& \widehat J_g^{\mathrm{ls}}(1)<J_g^{\max}.
	\end{cases}
\end{equation}
Here, $\widehat J_g^{\mathrm{ls}}(\alpha)$ denotes the average complete-trial
cost of $\theta_g(\alpha)$ over that stored scenario prefix. Equation
\eqref{eq:deterministic-nonmonotone-reference} retains the median/minimum/
maximum structure of the nonmonotone reference used in
Kimiaei and Neumaier~\cite{kimiaei2024madfo}, but the random convex-combination
parameter and random history replacement are removed. The retained history is
instead updated deterministically by the FIFO rule~\eqref{eq:ls-memory}.
Consequently, once the incoming history, generation data, and recombination
direction are fixed, the nonmonotone reference introduces no additional
randomization.
A larger $\eta_{\mathrm{nm}}$ gives more weight to the median in the first
branch and to the minimum in the second branch. The value of
$\eta_{\mathrm{nm}}$ is a tuning parameter and must be fixed before the
numerical experiment rather than selected after inspecting the results.

Because the present problem is chance constrained, a cost-only line search
must not be allowed to improve cost by silently worsening the empirical
constraint violation. For every line-search point define
\begin{equation}
	\label{eq:ls-risk-violation}
	v_g^{\mathrm{ls}}(\alpha)
	:=\max\left\{0,
	\widehat p_g^{\mathrm{ls}}(\alpha)-\delta
	\right\},
\end{equation}
where $\widehat p_g^{\mathrm{ls}}(\alpha)$ is computed from the same $K_g$
complete scenarios as $\widehat J_g^{\mathrm{ls}}(\alpha)$. These quantities
are search statistics only; they are not substitutes for the independent
final validation in Section~\ref{sec:validation}.

	The line-search statistics are not included in the simultaneous-coverage
	event of Proposition~\ref{prop:coverage}; they are used only to guide the
	optimization step. No finite-sample certification of the true failure
	probability or expected cost of a line-search point is claimed at this
	stage. Final risk certification is provided only by the independent
	validation procedure in Step~(S6).

\paragraph{Geometric search and extrapolation.}
	Choose fixed parameters $0<\gamma_{\mathrm{ls}}<1$, $\gamma_e>1$,
	an initial search multiplier $0<\alpha_0\leq1$, a maximum multiplier
	$\alpha_{\max}>1$, and a maximum number $Q_{\max}$ of line-search points.
	The standard \texttt{MA-ES} recombination point $\alpha=1$ is always
	evaluated and retained as a fallback. Independently, the geometric search
	starts from $\alpha_0$ and uses
	\begin{equation}
		\label{eq:ls-expansion}
		\alpha_{q+1}=\gamma_e\alpha_q,
		\qquad q=0,1,\ldots.
	\end{equation}
	Thus $\alpha_0<1$ provides an under-relaxed trial before the sequence enters
	the extrapolation region $\alpha>1$, whereas $\alpha_0=1$ recovers the
	original pure-extrapolation initialization. A duplicate evaluation at
	$\alpha=1$ is skipped. A non-nominal search trial is admissible for
	continuation if either
	\begin{equation}
		\label{eq:ls-risk-improvement}
		v_g^{\mathrm{ls}}(\alpha)
		< v_g^{\mathrm{ls}}(1),
	\end{equation}
	or
	\begin{equation}
		\label{eq:ls-cost-condition}
		v_g^{\mathrm{ls}}(\alpha)
		= v_g^{\mathrm{ls}}(1)
		\quad\text{and}\quad
		\widehat J_g^{\mathrm{ls}}(\alpha)
		+\gamma_{\mathrm{ls}}(\alpha\sigma_g)^2
		< J_g^{\mathrm{nm}}.
	\end{equation}
	Condition~\eqref{eq:ls-cost-condition} uses the deterministic nonmonotone
	reference term and is adapted from the forcing-function nonmonotone
	condition used for recombination extrapolation
	in~\cite{kimiaei2024madfo}. The forcing term contains no directional
	derivative. The quantity $\gamma_{\mathrm{ls}}(\alpha\sigma_g)^2$ is a prescribed
		derivative-free forcing term used only to control continuation of the
		extrapolation sequence; no convergence claim in this paper relies on a
		step-length interpretation of this term. Condition~\eqref{eq:ls-risk-improvement} preserves the
	risk-first structure already used by the empirical key~\eqref{eq:rank}:
	a strict reduction of empirical chance-constraint violation is preferred,
	whereas cost extrapolation is allowed only on the same empirical violation
	level. If a trial fails both continuation tests, the search is terminated after
	that trial has been evaluated, so the continuation-failing terminal point
	remains eligible for final selection. The upper bound $\alpha_{\max}$ is a
	separate safeguard: if the next geometric multiplier would exceed
	$\alpha_{\max}$, the point $\alpha=\alpha_{\max}$ is evaluated once and the
	search terminates because the cap has been reached. This capped point is not
	classified as a failed continuation point and also remains eligible for final
	selection. Let $\mathcal E_g$ denote the set of all multipliers actually
	evaluated during the call, including the nominal multiplier $1$, any
	under-relaxed and extrapolated search multipliers, and either type of terminal
	point just described. The conditions
	\eqref{eq:ls-risk-improvement}--\eqref{eq:ls-cost-condition} determine only
	whether extrapolation continues; they do not determine whether an already
	evaluated point remains eligible for final selection. Among all evaluated line-search points, choose
	\begin{equation}
		\label{eq:ls-best-alpha}
		\alpha_g^{\mathrm{ls}}
		\in\arg\min_{\alpha\in\mathcal E_g}
		\left(
		v_g^{\mathrm{ls}}(\alpha),
		\widehat J_g^{\mathrm{ls}}(\alpha),
		\alpha
		\right)
	\end{equation}
	lexicographically, where the last component gives deterministic tie
	breaking. The refined recombination point is then
	\begin{equation}
		\label{eq:ls-refined-mean}
		m_{g+1}^{\mathrm{ls}}
		:=m_g+\alpha_g^{\mathrm{ls}}\sigma_gM_g\bar z_g.
	\end{equation}
	Hence the line search can only retain the nominal \texttt{MA-ES} point or
	replace it by a selected extrapolated point along the same recombination
	direction. No random sign reversal is used. This is intentional:
	$+d_g^{\mathrm{rec}}$ points toward the rank-weighted selected population,
	whereas searching $-d_g^{\mathrm{rec}}$ would be a separate rescue mechanism
	rather than a refinement of the \texttt{MA-ES} recombination point.
	
\paragraph{Insertion into Step (S5).}
The line search is performed after the generation-$g$ race has stopped and
a single admissible order has been chosen, but before the next population
is sampled. The implementation should use the following order.
\begin{enumerate}[label=(S5$_{\arabic*}$),leftmargin=2.8em,itemsep=2pt,topsep=4pt]
	\item Compute $\bar z_g$ and $S_g$ from the chosen order exactly as in
	Section~\ref{sec:racing}, and compute the standard state updates
	$s_{g+1}$, $M_{g+1}$, and $\sigma_{g+1}$ from
	\eqref{eq:path}--\eqref{eq:scale}.
	\item Form $d_g^{\mathrm{rec}}=M_g\bar z_g$ and the nominal mean
	$m_{g+1}^{\mathrm{MA}}$ in~\eqref{eq:ls-recombination-direction}.
	\item Apply the line search only when the recombination direction is
	rank-certified. If it is not rank-certified, skip the line search and set
	$m_{g+1}=m_{g+1}^{\mathrm{MA}}$.
	\item When the recombination direction is rank-certified, evaluate the
	extrapolation sequence \eqref{eq:ls-expansion} on the stored
	generation-$g$ scenario prefix. Use
	\eqref{eq:ls-risk-improvement}--\eqref{eq:ls-cost-condition} only to decide
	whether extrapolation continues, retain every evaluated multiplier including
	the terminal one, and compute $\alpha_g^{\mathrm{ls}}$ over the resulting
	set $\mathcal E_g$ from~\eqref{eq:ls-best-alpha}.
	\item Set $m_{g+1}=m_{g+1}^{\mathrm{ls}}$ and leave
	$s_{g+1}$, $M_{g+1}$, and $\sigma_{g+1}$ unchanged from the standard
	\texttt{MA-ES} formulas. The line-search observations are not inserted
	back into the already completed generation-$g$ racing graph.
	\item Sample generation $g+1$ from the resulting state in the usual way.
\end{enumerate}
	Thus the line search changes only the next sampling center and not the
	rank weights, evolution path, transformation update, or scale update.
	The same stored scenarios and the same stored standard-normal \texttt{MPPI}
	base draws should be reused for all $\alpha$ tested in one line search, so
	that differences between trial points are not confounded by a changing
	scenario realization.
	
	\paragraph{Why rank certification preserves the present update theory.}
	The exact-direction gate is important. On the simultaneous coverage event,
	rank certification of the recombination direction implies
	$\widehat{\bar z}_g=\bar z_g^*$ by~\eqref{eq:stat-bounds}. Therefore the
	chosen-order and true-order recombination directions are identical. Since
	\eqref{eq:ls-expansion}, the deterministic nonmonotone memory and reference,
	the scenario prefix, and the tie breaking are fixed once the incoming history
	and generation information are given,
both orderings generate the same evaluated multiplier set $\mathcal E_g$
and the same selected multiplier $\alpha_g^{\mathrm{ls}}$. Consequently, for the line-search-refined mean, $E_z=0$ results in $\widehat m_{g+1}^{\mathrm{ls}}
		=m_{g+1}^{\mathrm{ls},*}$ on the same coverage event. The bounds for the path, matrix, and scale are
	unchanged because those three quantities continue to use the original
	\texttt{MA-ES} recursions.
	
	If the line search were instead allowed when $E_z>0$, the existing mean
	bound would no longer by itself control the refined mean. If
	$\widehat\alpha_g$ and $\alpha_g^*$ denote the multipliers obtained from
	the chosen and true orderings, respectively, then a direct decomposition
	gives
	\begin{equation}
		\label{eq:ls-uncertain-alpha-bound}
		\left\|
		\widehat m_{g+1}^{\mathrm{ls}}-m_{g+1}^{\mathrm{ls},*}
		\right\|_2
		\le
		\sigma_g\|M_g\|_2
		\left(
		\widehat\alpha_g E_z
		+|\widehat\alpha_g-\alpha_g^*|\,\|\bar z_g^*\|_2
		\right).
	\end{equation}
	The present theory contains no nontrivial bound on
	$|\widehat\alpha_g-\alpha_g^*|$, because a small perturbation of the
	recombination direction can change which extrapolation point first violates
	the line-search condition. We therefore do not use a tolerance such as
	``small $E_z$'' as a justification for the line search. The exact gate
	$E_z=0$ keeps the extension compatible with the current
	population-conditional rank-to-update analysis without introducing a new
	line-search stability assumption.
	
\begin{corollary}[Complete Step~(S5) discrepancy]
	\label{cor:s5-discrepancy}
	On the simultaneous coverage event, consider the chosen-order
	and true-order updates from the same incoming search state
	and with the same scenario and nonmonotone-memory history.
	
	(i) If $E_z>0$, the line search is skipped and the bounds of
	Proposition~\ref{prop:update} apply directly.
	
	(ii) If $E_z=0$, then $E_S=0$, and the chosen- and true-order
	recombination directions coincide. Under the prescribed
	deterministic line-search rule, both updates evaluate the same
	trial points and select the same multiplier. Consequently,
	\[
	\widehat m_{g+1}^{\mathrm{ls}}=m_{g+1}^{\mathrm{ls},*},
	\qquad
	\widehat s_{g+1}=s_{g+1}^*,
	\qquad
	\widehat M_{g+1}=M_{g+1}^*,
	\qquad
	\widehat\sigma_{g+1}=\sigma_{g+1}^*.
	\]
\end{corollary}
\begin{proof}
	Part (i) follows from the line-search activation rule and
	Proposition~\ref{prop:update}. For part (ii), each term
	$b_i\|z_i\|_2$ in the definition of $E_z$ is nonnegative.
	Thus $E_z=0$ implies $b_i\|z_i\|_2=0$ for every $i$,
	and consequently $E_S=0$. Proposition~\ref{prop:update}
	then gives equality of the standard recombination direction,
	evolution path, transformation matrix, and global scale.
	With a common incoming history and scenario prefix, the
	deterministic line-search construction evaluates the same
	points and selects the same multiplier under both orderings.
	The refined means therefore coincide.
\end{proof}
	
	\paragraph{Evaluation accounting.}
	The line search is derivative free but not evaluation free. Let $Q_g$ denote
	the total number of line-search points evaluated in generation $g$, including
	the nominal point $\alpha=1$ whenever the line search is activated. Each is
	evaluated on the same $K_g$ complete scenarios, so the line search adds
	$Q_gK_g$ complete robot trials.
	If $G\le G_{\max}$ generations are executed, then whenever the
	rank-certification gate activates this refinement, the total training cost
	must therefore be reported as
	\begin{equation}
		\label{eq:ls-training-budget}
		N_{\mathrm{train}}^{\mathrm{ls}}
		=
		\sum_{g=0}^{G-1}(\lambda+Q_g)K_g,
		\qquad 0\le Q_g\le Q_{\max},
	\end{equation}
	instead of attributing these evaluations to the statistical race. The
	independent development and validation data remain untouched. In
	particular, the line-search scenario reuse is an optimization device and
	must never be counted as independent safety validation.

\subsection{(S6): Independent selection and risk validation}
\label{sec:validation}

	The data used to tune a controller cannot also provide an independent final
	safety assessment. The goal of this step is therefore to select one controller,
	freeze all of its settings, and evaluate it without further adaptation. We use
	separate development data for selection and an untouched fixed-size test set
	for the final risk statement.

With $K_g$ scenarios used by the stopped race in generation $g$ and $Q_g$ line-search points evaluated in that generation, the complete training search uses
	\[
	N_{\mathrm{train}}
	=
	\sum_{g=0}^{G-1}(\lambda+Q_g)K_g,
	\qquad 0\le Q_g\le Q_{\max},
	\]
complete robot trials,
where $G\le G_{\max}$ and $K_g\le K_{L_{\mathrm{look}}}$. When the line search is not activated, $Q_g=0$, and this reduces to $\lambda\sum_{g=0}^{G-1}K_g$. Each complete trial entails at most $TRBH$ internal \texttt{MPPI} prediction-model steps. The statistical evaluator and \texttt{MA-ES} are offline; neither adds a step to deployed \texttt{MPPI}. Real-time operation still requires timing measurements on the target hardware.

Re-evaluate a predeclared shortlist on independent development scenarios
and select one controller. Then freeze its parameters, estimator settings,
failure definition, scenario law, and online sampling budget. Let $\omega'_1,\ldots,\omega'_N$ denote the final untouched test scenarios,
with fixed, predeclared size $N$. Then
\begin{equation}
	\label{eq:validation}
	\widehat J_{\mathrm{val}}=\frac1N\sum_{r=1}^NC(\theta,\omega'_r),\qquad
	F_N=\sum_{r=1}^N I_{\mathrm{fail}}(\theta,\omega'_r),\qquad
	\widehat p_{\mathrm{val}}=F_N/N.
\end{equation}
For $0<\alpha_v<1$, compute the one-sided Clopper--Pearson upper bound
\cite{clopper1934}
\begin{equation}
	\label{eq:cp-validation}
	U_N(f;\alpha_v)=\begin{cases}
		\operatorname{B}^{-1}(1-\alpha_v;f+1,N-f),& f<N,\\
		1,& f=N.
	\end{cases}
\end{equation}
Accept the stated chance constraint only when $U_N(F_N;\alpha_v)\le\delta$.

\begin{proposition}[Independent-test false acceptance]
	\label{prop:validation}
	If the frozen controller has true failure probability $p>\delta$ under
	the test law and the $N$ test scenarios are i.i.d. and independent of all
	selection data, then
	$\mathbb P\{U_N(F_N;\alpha_v)\le\delta\}\le\alpha_v$.
\end{proposition}
\begin{proof}
	Condition on the frozen controller. Exact one-sided binomial coverage gives
	$\mathbb P\{p>U_N(F_N;\alpha_v)\}\le\alpha_v$. When $p>\delta$, the
	acceptance event is contained in that noncoverage event. The same statement
	holds after averaging over any independent training and development
	selection procedure.
\end{proof}
For zero failures, $U_N(0;\alpha_v)=1-\alpha_v^{1/N}$. Thus even in that
most favorable case acceptance requires
\begin{equation}
	\label{eq:zero-failures}
	N\ge\left\lceil{\log\alpha_v}/{\log(1-\delta)}\right\rceil.
\end{equation}
This is a frequentist false-acceptance guarantee, not a posterior
probability that a controller is safe. It does not establish zero risk,
continuous-time collision avoidance, transfer from simulation to reality,
or validity under distribution shift. Testing repeatedly until a controller
passes is not covered by the fixed-$N$ rule. Reusing test data for tuning,
multiple-controller selection, or repeated claims requires a new
prespecified multiplicity or sequential-testing design.

These six steps complete the proposed tuning and validation procedure. We next instantiate the complete construction in a controlled robot-navigation study and examine what the available evaluation budget can and cannot demonstrate.


\section{Computational complexity of the proposed framework}
\label{sec:computational-complexity}

\paragraph{Motivation and objective.}
Computational complexity is particularly important in robotic control
because the controller must generate commands within a prescribed
sampling period, whereas offline parameter tuning may require many
complete closed-loop simulations. In the proposed framework,
\texttt{MPPI} performs sampling-based optimization at every control
instant, while the outer \texttt{MA-ES} procedure repeatedly evaluates
complete robot trials. Statistical racing and the nonmonotone line
search introduce additional computations and evaluations, but do not
alter the prescribed online \texttt{MPPI} sampling budget.

The objective of this section is to quantify the worst-case
computational effort required by these operations. We distinguish
the cost of one online control update, the cost of a complete
closed-loop trial, and the total cost of the offline tuning procedure.
The resulting bound characterizes arithmetic work under a finite
evaluation budget; it is not an evaluation-complexity bound for
attaining a globally optimal or approximately optimal controller.

\paragraph{Computational model and intuition.}
Let $T$ denote the maximum number of executed control instants
in a complete trial, $H$ the \texttt{MPPI} prediction horizon,
$B$ the number of sampled command sequences per refinement pass,
and $R$ the number of refinement passes at each control instant.
Let $N_o$ denote the maximum number of obstacles in a scenario,
and let $N_\Gamma$ denote the maximum number of reference-route
segments inspected during a route-distance evaluation.

The principal computational operation of \texttt{MPPI} is the
evaluation of sampled predicted trajectories. At each control
instant, the planner propagates $B$ trajectories of length $H$
during each of its $R$ refinement passes. Consequently, at most
$RBH$ sampled prediction transitions are evaluated per control
instant and $TRBH$ per complete closed-loop trial.

For the fixed-dimensional kinematic robot models considered here,
assume that one predicted transition, including its planning score,
obstacle-clearance evaluation, and route-distance computation,
requires at most
$\mathcal{O}(N_o+N_\Gamma+1)$ arithmetic operations.
We assume the same upper bound for an executed robot transition,
including collision checking and evaluation of the outer task cost.
These assumptions correspond to constant-cost kinematic state
updates and direct processing of the relevant obstacle and
route geometry.

The offline optimization procedure performs at most $G_{\max}$
generations with population size $\lambda$, at most
$L_{\mathrm{look}}$ statistical racing looks per generation,
and at most $K_{\max}$ complete trials per candidate.
The standard \texttt{MA-ES} search dimension is $n$.
Let $m_{\mathrm{nm}}\geq1$ denote the prescribed
nonmonotone-memory length, measured in generations.

The total complete-trial budget must include both the population
evaluations and the additional points evaluated by the
nonmonotone line search. We use the evaluation-accounting
convention already established in
\eqref{eq:ls-training-budget}, and let $E_{\max}$ be a hard
upper bound on that total training count. In particular,
$N_{\mathrm{train}}^{\mathrm{ls}}\leq E_{\max}$.
Independent development and final-validation trials are
excluded from this training budget.

\begin{proposition}[Worst-case computational complexity]
	\label{prop:computational-complexity}
	Consider the proposed framework under the computational
	model specified above. Assume that the robot state and
	control dimensions are fixed, and that predicted and executed
	robot transitions, including obstacle and route processing,
	require at most
	$\mathcal{O}(N_o+N_\Gamma+1)$ arithmetic operations.
	
	Assume further that the \texttt{MA-ES} transformation-matrix
	update is evaluated using conventional dense matrix
	operations, that pairwise racing statistics are recomputed
	directly at each look, and that the comparison graph is
	processed using a cubic-time transitive-closure algorithm.
	Let $C_{\mathrm{CP}}$ denote an upper bound on the
	arithmetic work required to compute one numerical
	Clopper--Pearson confidence endpoint using the
	prescribed numerical tolerance and quantile algorithm.
	
	Then the worst-case arithmetic complexity of the offline
	training procedure is bounded by
	\begin{equation}
		\label{eq:total-computational-complexity}
		\begin{aligned}
			\mathcal{O}\Bigl(
			&E_{\max}\bigl[TRBH(N_o+N_\Gamma+1)+n\bigr]+G_{\max}\Bigl[
			\lambda n^2+n^3\\
			&\quad+L_{\mathrm{look}}
			\bigl(\lambda^2K_{\max}+\lambda^3+n^3
			+\lambda C_{\mathrm{CP}}\bigr)+m_{\mathrm{nm}}\lambda
			\log(m_{\mathrm{nm}}\lambda+1)
			\Bigr]\Bigr).
		\end{aligned}
	\end{equation}
	In particular, the arithmetic complexity of one online
	\texttt{MPPI} control update is
	\begin{equation}
		\label{eq:mppi-online-complexity}
		\mathcal{O}\bigl(RBH(N_o+N_\Gamma+1)\bigr).
	\end{equation}
\end{proposition}

\begin{proof}
	At each control instant, \texttt{MPPI} performs $R$
	sampling/update passes, each propagating $B$ predicted
	trajectories over $H$ steps. Each predicted transition
	and its associated planning-score evaluation require at
	most $\mathcal{O}(N_o+N_\Gamma+1)$ operations.
	Gaussian perturbation generation, importance-weight
	computation, and nominal-command updating require
	$\mathcal{O}(RBH)$ additional operations because the
	control dimension is fixed. This establishes
	\eqref{eq:mppi-online-complexity}.
	
	A complete closed-loop trial contains at most $T$
	control instants. The executed-state propagation,
	collision checking, and outer-cost evaluation are
	absorbed into the per-control-instant bound. Hence,
	one complete trial requires at most
	$\mathcal{O}(TRBH(N_o+N_\Gamma+1))$ operations.
	By the hard training-budget assumption, the total
	robotic-simulation cost is bounded by
	$\mathcal{O}(E_{\max}TRBH(N_o+N_\Gamma+1))$.
	
	Within each generation, drawing and transforming
	$\lambda$ Gaussian mutation vectors requires
	$\mathcal{O}(\lambda n^2)$ operations. Forming the
	weighted second-moment matrix requires at most
	$\mathcal{O}(\lambda n^2)$ additional operations.
	The standard \texttt{MA-ES} mean, path, matrix, and
	scale recursions are given in
	\eqref{eq:mean}--\eqref{eq:scale}.
	Evaluating the transformation-matrix recursion
	\eqref{eq:matrix} through a conventional dense
	matrix--matrix multiplication requires
	$\mathcal{O}(n^3)$ operations.
	The remaining state updates are absorbed into
	these bounds. Thus, the standard per-generation
	adaptation cost is bounded by
	$\mathcal{O}(\lambda n^2+n^3)$.
	
	At each racing look, directly recomputing all
	pairwise cost-difference statistics using at most
	$K_{\max}$ observations per candidate requires
	$\mathcal{O}(\lambda^2K_{\max})$ operations.
	Computing the transitive closure of the comparison
	graph requires $\mathcal{O}(\lambda^3)$ operations.
	The admissible rank intervals, chosen ordering,
	and update-discrepancy bounds can be computed within
	these bounds, apart from the spectral-norm calculation.
	
	The spectral norm of the current transformation matrix,
	which enters the update-discrepancy bounds, can be
	computed by a standard singular-value decomposition
	in $\mathcal{O}(n^3)$ operations.
	Computing the risk-confidence endpoints for the
	$\lambda$ candidates requires at most
	$\mathcal{O}(\lambda C_{\mathrm{CP}})$ operations.
	Consequently, the racing overhead per generation is
	bounded by
	\[
	\mathcal{O}\left(
	L_{\mathrm{look}}
	\left[
	\lambda^2K_{\max}+\lambda^3+n^3
	+\lambda C_{\mathrm{CP}}
	\right]
	\right).
	\]
	
	The nonmonotone reference defined in
	\eqref{eq:ls-memory}--\eqref{eq:ls-min-med-max}
	contains at most $m_{\mathrm{nm}}\lambda$
	stored population-cost values. Computing the
	required minimum, median, and maximum by sorting
	requires at most
	\[
	\mathcal{O}\left(
	m_{\mathrm{nm}}\lambda
	\log(m_{\mathrm{nm}}\lambda+1)
	\right)
	\]
	operations per generation.
	
	Each evaluated line-search point defined by
	\eqref{eq:ls-trial-point} requires
	$\mathcal{O}(n)$ additional arithmetic operations
	to form its latent coordinates and transform them
	to admissible controller parameters.
	
	By the evaluation accounting in
	\eqref{eq:ls-training-budget}, and since each
	line-search point is evaluated on $K_g\geq1$
	complete scenarios,
	\[
	\sum_{g=0}^{G-1}Q_g
	\leq
	\sum_{g=0}^{G-1}Q_gK_g
	\leq E_{\max}.
	\]
	Therefore, the total line-search point-construction
	overhead is bounded by $\mathcal{O}(E_{\max}n)$.
	The corresponding closed-loop evaluations are
	already counted in the robotic-simulation cost.
	
	Finally, summing the adaptation, racing, and
	nonmonotone-memory overhead over at most
	$G_{\max}$ generations and adding the total
	simulation and line-search point-construction costs
	establishes \eqref{eq:total-computational-complexity}.
\end{proof}

\paragraph{Interpretation and practical implications.}
The proposition separates the computation required by
the online controller from the additional computation
introduced by offline parameter optimization.
Statistical racing and nonmonotone recombination
refinement do not change the prescribed number of
\texttt{MPPI} samples, refinement passes, or prediction
steps at an individual control instant. Instead,
they affect how the available complete-trial budget
is distributed among population evaluations and
additional recombination trial points.

For the exact-budget numerical benchmark,
$T=35$, $H=8$, $B=18$, and $R=1$. Thus, at most
\[
RBH=144
\]
sampled prediction transitions are evaluated per
control instant, and
\[
TRBH=5040
\]
per complete closed-loop trial. Under the common
training budget $E_{\max}=432$, the maximum number
of sampled prediction transitions during training is
\[
E_{\max}TRBH=2\,177\,280
\]
per solver--model run. These are counts of predicted
robot transitions, not processor operations or
measured runtimes. Early trial termination may
reduce the actual number of transitions evaluated.

For the six-dimensional tuning problem considered
here, conventional dense matrix operations are
relatively inexpensive compared with repeated
closed-loop simulations. Nevertheless, the
complexity bound explicitly retains the cubic
dependence on $n$ arising from dense matrix
multiplication and spectral-norm computations.
Exploiting the rank-one structure of the standard
matrix-adaptation recursion could reduce part of
this overhead without changing the algorithmic
recursion itself.

Finally, this analysis concerns the computational
work of the finite-budget offline training
procedure. Independent development and
final-validation trials contribute additional
simulation work. The bound does not establish
real-time feasibility on a particular computing
platform, a worst-case optimization convergence
rate, global optimality, or a prescribed
optimality gap.

\section{Numerical study}
\label{sec:simulation}

We organize the numerical study around three outer-loop solvers and three
kinematic ground-vehicle models. The solvers are standard \texttt{MA-ES},
the same \texttt{MA-ES} recursion augmented by the frozen nonmonotone
line-search variant (\texttt{MA-ES+NMLS}), and \texttt{VRBBO}~3.2,
a randomized black-box line-search method~\cite{kimiaei2021vrbbo}.
The models are \texttt{TurtleBot~4}, a Jackal-like skid-steer abstraction,
and a kinematic \texttt{Ackermann} bicycle model. The principal
comparison uses the same eleven optimization seeds and the same hard
budget of 432 complete training trials for every solver--model run.
We first describe the common experimental protocol and specify all
\texttt{MA-ES} and \texttt{NMLS} settings, then present the earlier
update-aware proof-of-execution and large-budget \texttt{NMLS} studies,
and finally report the new three-solver benchmark. At the end of
the section, we provide performance-metric-specific recommendations
rather than a single unconditional ranking of the solvers.

\subsection{Representative post-selection trajectories for the three vehicle models}
\label{sec:statall-model-trajectories}

The trajectory panels provide geometric illustrations of the closed-loop
behavior and are not used in any development or final-validation statistic.
For each vehicle model, the eleven frozen model-specific
\texttt{MA-ES+NMLS} controllers from the exact-budget benchmark are replayed,
without any further optimizer update, on the same fixed fresh illustration
protocol with an 80-step horizon. Among the collision-free replays that enter
the declared goal tolerance, the displayed controller is selected as the one
whose held-out cost is closest to the corresponding model-wise
\texttt{MA-ES+NMLS} median. This procedure selects optimization seed~3 for
\texttt{TurtleBot~4}, seed~37 for the \texttt{Jackal-like} model, and seed~3
for the \texttt{Ackermann} model. Whenever an executed trajectory first enters
the declared goal tolerance, only the genuinely executed prefix through that
first entry is shown. No trajectory segment is extrapolated, smoothed, or
otherwise fabricated.

In every panel, the horizontal and vertical coordinates represent the robot
positions $x$ and $y$ in meters, respectively. The solid blue curve is the true
executed closed-loop trajectory, while the orange dashed polyline is the fixed
reference route supplied to the \texttt{MPPI} controller. The blue filled
circle marks the initial position and the orange star denotes the goal. The
solid black circle represents the physical obstacle, whereas the dashed black
circle is the corresponding collision boundary obtained by inflating the
obstacle by the robot radius. The robot radii are 0.245~m for
\texttt{TurtleBot~4}, 0.300~m for the \texttt{Jackal-like} model, and
0.245~m for the \texttt{Ackermann} model. All three displayed trajectories are
collision-free and enter their declared goal tolerance. Panels~(a)--(c)
correspond to frozen \texttt{MA-ES+NMLS} controllers from seeds~3, 37, and~3,
respectively, replayed on the fixed 80-step illustration scenario. Their
corresponding illustrative-trial costs are 41.16, 25.29, and 46.60. These
trajectories are included only for qualitative geometric illustration and do
not contribute to the development data, final-validation statistics, or
three-solver statistical comparisons.

\begin{figure}[t]
	\centering
	\begin{minipage}[t]{0.32\linewidth}
		\centering
		\includegraphics[width=\linewidth]
		{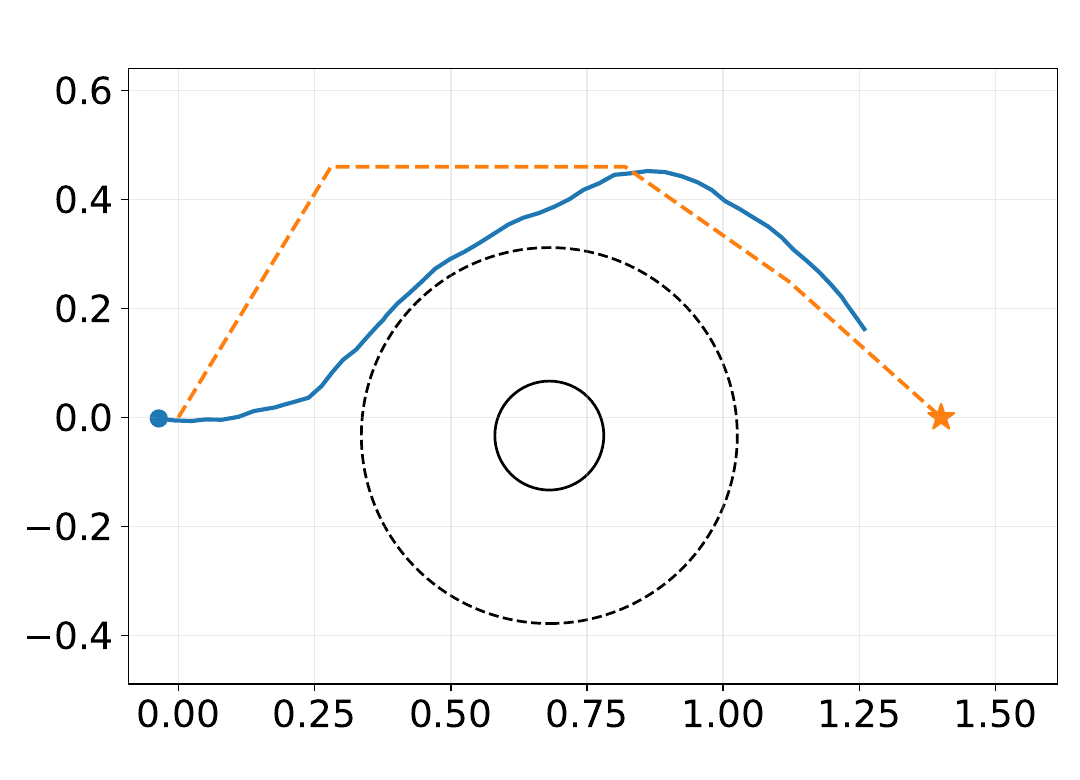}
		\small (a) \texttt{TurtleBot~4}: seed 3, illustrative-trial cost 41.16.
	\end{minipage}\hfill
	\begin{minipage}[t]{0.32\linewidth}
		\centering
		\includegraphics[width=\linewidth]
		{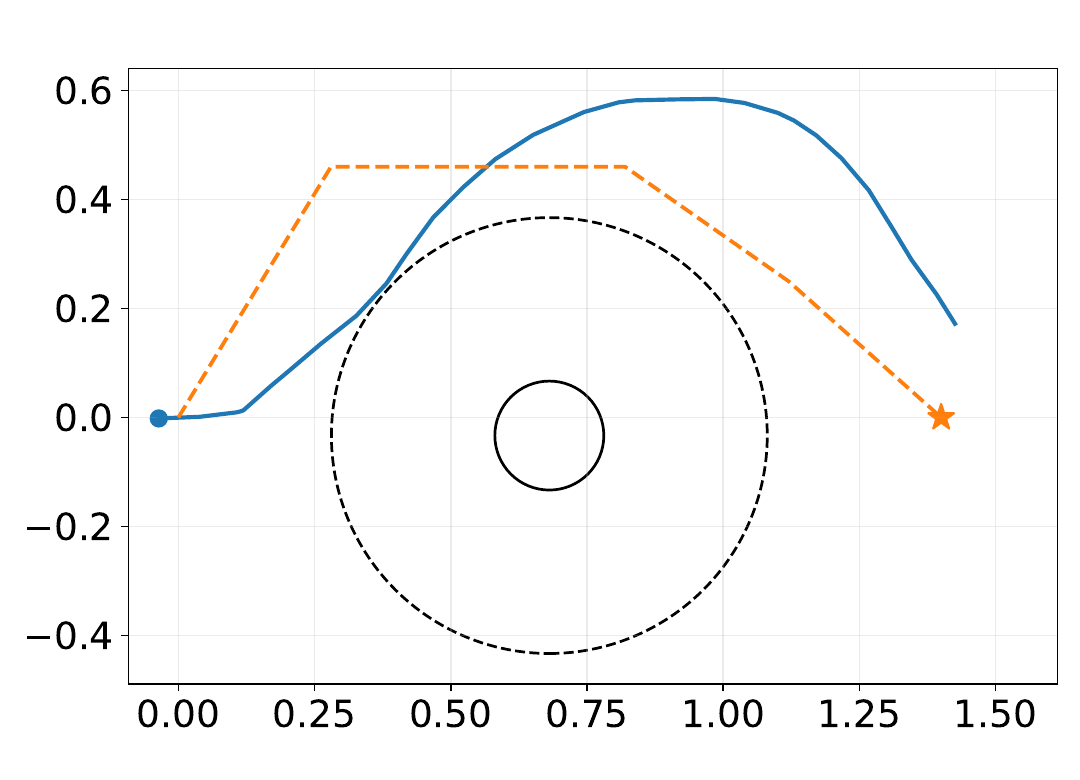}
		\small (b) \texttt{Jackal-like}: seed 37, illustrative-trial cost 25.29.
	\end{minipage}\hfill
	\begin{minipage}[t]{0.32\linewidth}
		\centering
		\includegraphics[width=\linewidth]
		{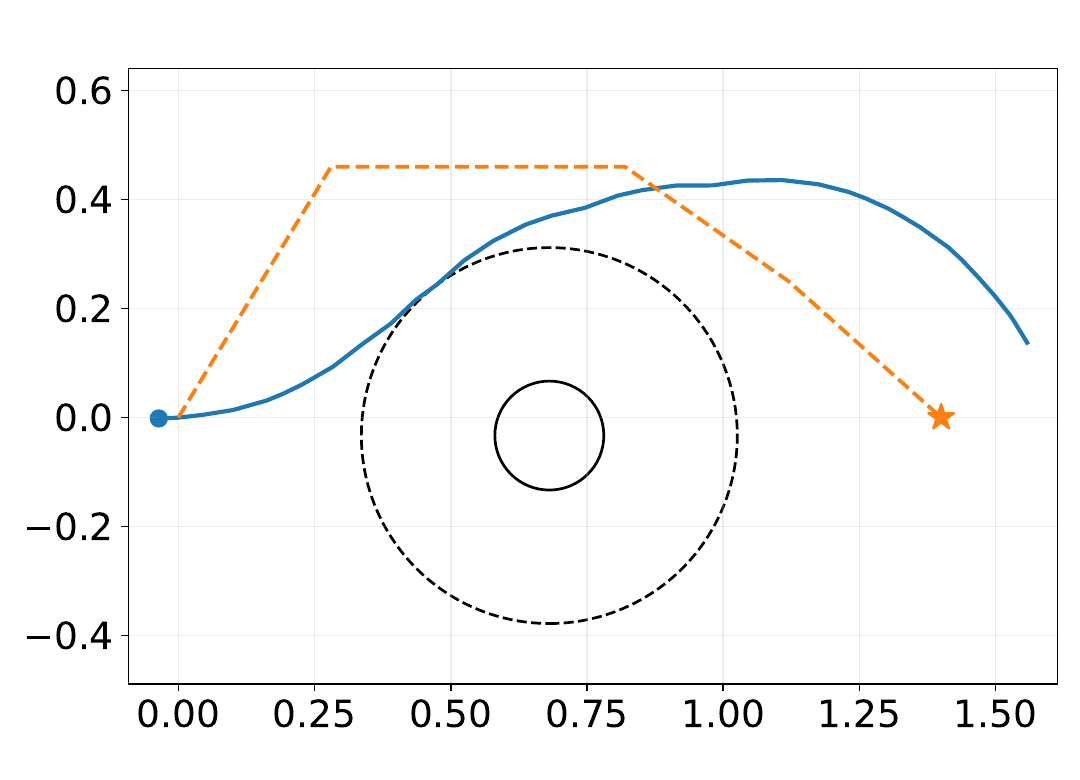}
		\small (c) \texttt{Ackermann}: seed 3, illustrative-trial cost 46.60.
	\end{minipage}
	
	\caption{Representative post-selection closed-loop trajectories for the
		three vehicle models on the common one-obstacle illustration world.
		The solid blue curve denotes the executed robot trajectory, the orange
		dashed polyline denotes the fixed reference route, the blue filled circle
		marks the initial position, and the orange star marks the goal. The solid
		black circle is the physical obstacle and the dashed black circle is the
		inflated collision boundary.}
	\label{fig:statall-three-model-trajectories}
\end{figure}

\subsection{Implementation, models, and common experimental protocol}
\label{sec:simsetup}

The \texttt{MPPI}, \texttt{MA-ES}, \texttt{NMLS}, validation, and statistical analysis
components are implemented in Python using NumPy and SciPy. \texttt{VRBBO}~3.2 is run
from its original \texttt{MATLAB} implementation, with a Python bridge used only to
call the same controller simulator and validation protocol. The simulator is
standalone: it does not include a physics engine, Robot Operating System
(\texttt{ROS}) middleware, simultaneous localization and mapping
(\texttt{SLAM}), or a hardware localization stack. Unit tests check the bounded
parameter map, finite \texttt{MA-ES} updates, exact zero-failure
Clopper--Pearson identity, swept-disk collision geometry, \texttt{MPPI}
reproducibility and actuator bounds, deterministic complete-trial
reproducibility, and the analytic outer cost bound. These tests establish
software consistency; they do not replace the analytical statistical
guarantees developed above.

For each candidate and scenario, the robot state, previous command, estimator
perturbations, and \texttt{MPPI} warm start are reset. Candidates compared in
a common \texttt{MA-ES} population use common random numbers. The same latent
six-dimensional parameterization and the same bounded physical parameter map
are used in all three solvers. In the budget-matched benchmark, every scalar
training evaluation uses $K=4$ complete closed-loop trials, development
selection uses eight fresh trials, and final evaluation uses 32 independent
trials. For this benchmark, the chance-constraint threshold and the
final-validation noncoverage level were fixed before all optimization
runs at $\delta=0.15$ and $\alpha_v=0.05$, respectively. These
same values were used to determine all validation-acceptance entries
reported in the ``Acc.'' column of Table~\ref{tab:three-solver-main}.
The eleven common outer seeds are
\[
\{3,7,11,17,23,29,31,37,41,47,53\}.
\]

The three models are summarized in Table~\ref{tab:three-models}. \texttt{TurtleBot~4}
uses the unicycle model already defined in~\eqref{eq:unicycle} with the
manufacturer-reported command limits used in the earlier study~\cite{turtlebot4manual}. The \texttt{Jackal-like}
model uses the same kinematic form with a larger footprint and higher speed.
The \texttt{Ackermann} benchmark uses a bicycle heading update
\[
\psi_{k+1}=\psi_k+\Delta t\,\frac{v_k}{L_w}\tan\delta_{s,k}.
\]
The two transfer models are controlled kinematic abstractions rather than
manufacturer-specific high-fidelity dynamics.

\begin{table}[t]
	\centering
	\caption{Three vehicle models used in the common benchmark. The second control
		channel is yaw rate for the two unicycle models and steering angle for the
		\texttt{Ackermann} model.}
	\label{tab:three-models}
	\small
	\begin{tabular}{@{}lccccc@{}}
		\toprule
		Model & Kinematics & $v_{\max}$ & Second-control bound & Radius & Wheelbase \\
		\midrule
		\texttt{TurtleBot~4} & unicycle & $0.31$ m/s & $1.90$ rad/s & $0.245$ m & -- \\
		\texttt{Jackal-like} & unicycle & $0.80$ m/s & $1.50$ rad/s & $0.300$ m & -- \\
		\texttt{Ackermann} & bicycle & $0.65$ m/s & $0.60$ rad & $0.245$ m & $0.34$ m \\
		\bottomrule
	\end{tabular}
\end{table}

For the exact-budget benchmark, all models use $T=35$ executed steps,
\texttt{MPPI} horizon $H=8$, $B=18$ sampled command sequences, and one
sampling/update pass per control instant. The outer hard training budget is
432 complete trials. Standard \texttt{MA-ES} therefore fits 12 full
population generations ($9\times4$ trials per generation), while
\texttt{MA-ES+NMLS} may execute fewer generations because line-search points
consume part of the same hard budget. \texttt{VRBBO} uses $n_f^{\max}=108$ scalar
objective evaluations, and each such value uses four complete trials, giving
exactly the same $108\times4=432$ complete-training-trial cap. \texttt{VRBBO} is
optimized in the same latent space; its scalar oracle encodes the same
finite-$K$ feasibility-first ordering, with any one-level risk improvement
dominating the bounded cost range. Consecutive groups of nine \texttt{VRBBO} calls use
a common scenario block, matching the \texttt{MA-ES} population size for the
common-random-number comparison.

\subsection{Solver settings and tuning parameters}
\label{sec:solver-settings}

Both \texttt{MA-ES} variants use the same six-dimensional search and the same
standard matrix-adaptation coefficients. For $n=6$, the standard formulas give
$\lambda=9$, $\mu=4$, recombination weights
\[
(w_1,w_2,w_3,w_4)
=(0.493738,0.281097,0.156710,0.068455),
\]
$\mu_{\rm eff}=2.840610$, $c_s=0.349740$, $c_1=0.035631$,
$c_\mu=0.035686$, $d_\sigma=1.349740$, and
$\chi_6=2.349964$. The initial global scale is $\sigma_0=0.55$, with
$M_0=I_6$ and $s_0=0$. The physical starting vector is
\[
\theta_0=(2,50,0.003,0.06,0.30,1.00),
\]
and the lower/upper bounds are
\[
\theta^-=(0.05,1,10^{-3},0.05,0.01,0.05),\qquad
\theta^+=(20,250,5,5,0.35,1.80).
\]
Table~\ref{tab:maes-settings} collects the settings that differ between the
two \texttt{MA-ES} arms. No model-specific retuning is performed.

\begin{table}[t]
	\centering
	\caption{Frozen optimizer settings in the exact-budget benchmark. Standard
		\texttt{MA-ES} and \texttt{MA-ES+NMLS} share the same matrix-adaptation
		parameters; \texttt{NMLS} changes only the next mean through the selected extrapolation
		multiplier.}
	\label{tab:maes-settings}
	\small
	\begin{tabular}{@{}lcc@{}}
		\toprule
		Setting & \texttt{MA-ES} & \texttt{MA-ES+NMLS} \\
		\midrule
		Search dimension $n$ & 6 & 6 \\
		Population / parents $(\lambda,\mu)$ & $(9,4)$ & $(9,4)$ \\
		$\mu_{\rm eff}$ & 2.840610 & 2.840610 \\
		$c_s$ & 0.349740 & 0.349740 \\
		$c_1$ & 0.035631 & 0.035631 \\
		$c_\mu$ & 0.035686 & 0.035686 \\
		$d_\sigma$ & 1.349740 & 1.349740 \\
		$\chi_6$ & 2.349964 & 2.349964 \\
		Initial scale $\sigma_0$ & 0.55 & 0.55 \\
		Initial $M_0,s_0$ & $I_6,0$ & $I_6,0$ \\
		Hard complete-trial budget & 432 & 432 \\
		Fixed training replications $K$ & 4 & 4 \\
		Maximum generations & 50 & 50 \\
		\texttt{NMLS} enabled & no & yes \\
		$\eta_{\rm nm}$ & -- & 0.90 \\
		$\gamma_{\rm ls}$ & -- & $10^{-4}$ \\
		$\gamma_e$ & -- & 1.50 \\
		$\alpha_0$ & -- & 0.50 \\
		$\alpha_{\max}$ & -- & 100 \\
		$Q_{\max}$ & -- & 64 \\
		Nonmonotone memory & -- & 4 generations \\
		Interpolation/backtracking & -- & disabled \\
		\bottomrule
	\end{tabular}
\end{table}

For \texttt{VRBBO}~3.2 we use the solver's \emph{default tuning values}; no internal
\texttt{VRBBO} tuning parameter is modified or tuned for these experiments. The only
external solver limit is $n_f^{\max}=108$, imposed solely to match the common
432-complete-trial budget. The \texttt{MATLAB} random-number generator is reset with
the same outer seed used to identify each optimization run.

\subsection{Retained update-aware proof-of-execution study}
\label{sec:results}

The original quick-budget \texttt{TurtleBot~4} experiment is retained because it
exercises the confidence-supported risk-to-rank-to-update construction rather
than only the optimizer comparison. It uses ten seeds
$\{3,7,11,17,23,29,31,37,41,47\}$, two \texttt{MA-ES} generations,
racing looks $K\in\{2,4\}$, six development trials, and 16 independent test
trials. For this update-aware study, all statistical and
stopping parameters were fixed before execution.
Specifically, the maximum number of generations
was $G_{\max}=2$, the search noncoverage level was
$\alpha_s=0.10$, and the chance-constraint
threshold was $\delta=0.15$.
The update tolerances were fixed at
\[
(\eta_m,\eta_s,\eta_M,\eta_{\log\sigma})
=
(0.08,0.10,0.12,0.05).
\]
Relative to the prescribed pre-tuning vector, the selected controllers
reduce the mean held-out cost from $75.08$ to $71.93$. The ten run-level paired
mean differences have mean $-3.15$ and 95\% Student-$t$ interval
$[-5.79,-0.51]$; nine of ten run-level differences are negative. Mean goal
completion rises from $0.219$ to $0.331$, with run-level mean-difference
interval $[0.0187,0.2063]$. Neither the selected nor initial controllers
record a collision in their paired 16-trial tests.

The main conclusion of that quick study is diagnostic rather than comparative:
all 20 generations reach the final look unresolved, so no generation satisfies
the rank-certification gate required to activate \texttt{NMLS}. At $K=4$, the allocated
Clopper--Pearson risk intervals cannot certify either feasibility or violation
separation, and the paired Hoeffding cost bound is also vacuous for the declared
global cost width. Thus the quick study verifies execution of the update-aware
logic but does not provide evidence of racing sample-efficiency or of the line
search itself.

\subsection{Retained large-budget \texttt{MA-ES} versus \texttt{NMLS} study}
\label{sec:maes-nmls-ablation}

We also retain the earlier eleven-seed \texttt{TurtleBot~4} large-budget experiment. In
that study, standard \texttt{MA-ES} uses 12 generations and 864 complete
population trials; the \texttt{NMLS} arm uses nine generations and a mean of 846.5
complete training trials because the number of line-search evaluations varies
by seed. The \texttt{NMLS} line search is deliberately activated each generation, so
this experiment isolates the algorithmic effect of the recombination search
rather than the rank-certified Step~(S5) gate. Its completed configuration is
$\eta_{\rm nm}=0.90$, $\gamma_{\rm ls}=10^{-4}$, $\gamma_e=1.50$,
$\alpha_0=1$, $\alpha_{\max}=100$, $Q_{\max}=32$, and
$m_{\rm nm}=4$.

\begin{table}[t]
	\centering
	\caption{Retained eleven-seed large-budget \texttt{TurtleBot~4} comparison. Lower cost
		is better. This study predates the final $\alpha_0=0.5$ transfer preset and
		uses $\alpha_0=1$.}
	\label{tab:retained-large-budget}
	\small
	\begin{tabular}{@{}lcc@{}}
		\toprule
		Quantity & \texttt{MA-ES} & \texttt{MA-ES+NMLS} \\
		\midrule
		Mean held-out cost & 56.3892 & \textbf{55.2408} \\
		Median held-out cost & 62.4401 & \textbf{54.4269} \\
		Best held-out cost & \textbf{39.6413} & 40.0070 \\
		Worst held-out cost & 72.2224 & \textbf{69.4405} \\
		Sample standard deviation & 12.1432 & 12.2301 \\
		Mean goal-completion rate & 0.6383 & \textbf{0.7027} \\
		\texttt{NMLS} better/equal/worse & \multicolumn{2}{c}{$6/0/5$} \\
		Paired \texttt{NMLS}$-$MA-ES mean & \multicolumn{2}{c}{$-1.1484$} \\
		Paired 95\% interval & \multicolumn{2}{c}{$[-7.7169,5.4201]$} \\
		\bottomrule
	\end{tabular}
\end{table}

This retained experiment is descriptively favorable to \texttt{NMLS}: the mean falls by
about $2.04\%$, the median by $12.83\%$, and the worst observed cost by
$3.85\%$. However, the paired interval contains zero and the sample standard
deviations are nearly identical. It therefore supports competitiveness and a
possible large-budget advantage, not statistically significant universal
superiority.

\subsection{Exploratory tuning of the frozen \texttt{NMLS} preset}
\label{sec:nmls-tuning}

The final transfer preset was selected before the \texttt{Jackal-like} and \texttt{Ackermann}
benchmarks and is not retuned by model. The forcing parameter screen
$\gamma_{\rm ls}\in\{10^{-12},10^{-4},10^{-2},0.9\}$ produced the same mean
held-out cost $67.0593$ on the tuning seeds, so we retain the conservative
choice $\gamma_{\rm ls}=10^{-4}$. A disjoint joint screen used
\[
\alpha_0\in\{10^{-3},10^{-2},10^{-1},0.25,0.5,1\},\qquad
\alpha_{\max}\in\{10^2,10^3,10^4\},
\]
on tuning seeds $\{101,103,107,109,113\}$. The mean held-out costs for the six
$\alpha_0$ values, averaged over the three caps, were
$42.8260$, $42.8281$, $43.0122$, $42.8242$, $42.8116$, and $44.5462$,
respectively. Thus $\alpha_0=0.5$ is the best tested value. The three
$\alpha_{\max}$ values are tied because the continuation test stops before the
cap becomes active; a separate eleven-seed sensitivity study reaches the same
conclusion, with largest evaluated multiplier $17.0859$. We therefore retain
$\alpha_{\max}=100$ as the smallest tested conservative cap. Finally, memory
lengths $m_{\rm nm}=1,2,4$ are tied in held-out cost on the memory screen, while
$m_{\rm nm}=4$ requires slightly fewer line-search trials. The resulting frozen
nonmonotone line-search preset is
\[
\boxed{\eta_{\rm nm}=0.90,\quad \gamma_{\rm ls}=10^{-4},\quad
	\gamma_e=1.50,\quad \alpha_0=0.50,\quad \alpha_{\max}=100,\quad
	Q_{\max}=64,\quad m_{\rm nm}=4.}
\]
Interpolation/backtracking is disabled. These values are benchmark-tuned
software settings and are not claimed to be universally optimal.

\subsection{Budget-matched three-solver benchmark}
\label{sec:multimodel-benchmark}

The main new comparison applies standard \texttt{MA-ES}, the frozen
\texttt{MA-ES+NMLS} variant, and \texttt{VRBBO} to all three models under exactly the
same 432-complete-trial training cap. Table~\ref{tab:three-solver-main}
reports the across-seed held-out statistics. Each row summarizes eleven
independent optimization runs, and every optimized controller is tested on 32
fresh final scenarios. Boldface identifies the best value within each model
for a given measure.

\begin{table}[t]
	\centering
	\caption{Budget-matched eleven-seed comparison of standard \texttt{MA-ES},
		\texttt{MA-ES+NMLS}, and \texttt{VRBBO}. Lower mean, median, best, worst, standard
		deviation, and failure count are better; larger validation acceptance is
		better. ``Acc.'' denotes the number of runs satisfying the fixed-sample
		validation rule.}
	\label{tab:three-solver-main}
	\scriptsize
	\resizebox{\linewidth}{!}{%
		\begin{tabular}{@{}llrrrrrrr@{}}
			\toprule
			Model & Solver & Mean & Median & Best & Worst & Std. & Fail. & Acc. \\
			\midrule
			\texttt{TurtleBot~4} & \texttt{MA-ES} & 50.5735 & 49.4997 & 46.9111 & \textbf{56.0679} & \textbf{3.2477} & 18 & 6/11 \\
			& \texttt{MA-ES}+\texttt{NMLS} & \textbf{50.0697} & \textbf{47.5745} & \textbf{46.4424} & 60.0585 & 4.7122 & \textbf{13} & \textbf{8/11} \\
			& \texttt{VRBBO} & 55.0754 & 55.1332 & 47.4499 & 67.3938 & 5.6259 & 22 & 7/11 \\
			\addlinespace
			\texttt{Jackal-like} & \texttt{MA-ES} & 51.2775 & 50.9758 & 46.5686 & \textbf{55.4898} & \textbf{2.8913} & 19 & 6/11 \\
			& \texttt{MA-ES}+\texttt{NMLS} & \textbf{51.1620} & \textbf{49.9713} & \textbf{46.2717} & 57.5392 & 4.1542 & \textbf{8} & \textbf{10/11} \\
			& \texttt{VRBBO} & 54.1265 & 53.1709 & 47.8246 & 59.0920 & 4.2203 & 14 & 6/11 \\
			\addlinespace
			\texttt{Ackermann} & \texttt{MA-ES} & 49.4768 & \textbf{47.1382} & 45.8879 & 60.2530 & 4.3500 & 12 & 7/11 \\
			& \texttt{MA-ES}+\texttt{NMLS} & \textbf{49.2674} & 49.4120 & \textbf{45.8598} & \textbf{53.8389} & \textbf{2.3689} & \textbf{5} & \textbf{10/11} \\
			& \texttt{VRBBO} & 53.5295 & 53.8482 & 46.2714 & 62.5728 & 5.0883 & 13 & 8/11 \\
			\bottomrule
	\end{tabular}}
\end{table}

\texttt{NMLS} has the lowest mean held-out cost on every model. On \texttt{TurtleBot~4} it also
has the lowest median and best cost and fewer failures, while standard
\texttt{MA-ES} has the lower worst-case cost and lower across-seed standard
deviation. On the \texttt{Jackal-like} model the pattern is similar: \texttt{NMLS} leads mean,
median, best cost, failures, and validation acceptance, whereas standard
\texttt{MA-ES} has the better worst-case cost and lower variability. On
\texttt{Ackermann}, \texttt{NMLS} leads mean, best, worst, standard deviation, failure count,
and validation acceptance; standard \texttt{MA-ES} has the better median.
\texttt{VRBBO} does not attain the best value for any of the pooled performance or
risk measures in Table~\ref{tab:three-solver-main} under the matched budget.

Table~\ref{tab:three-solver-paired} gives paired differences on the eleven
common seeds. A negative difference means that the first-named solver has lower cost. For the \texttt{NMLS}--\texttt{MA-ES} comparison, every model-wise interval
contains zero. Hence the new benchmark does not establish a statistically
significant cost difference between the two \texttt{MA-ES} variants. The
\texttt{VRBBO}-minus-\texttt{NMLS} intervals are positive for \texttt{Jackal-like} and \texttt{Ackermann}, while
the TurtleBot interval narrowly includes zero. The model-wise sign tests provide complementary
information about the frequency of pairwise wins
and should be interpreted alongside the paired
confidence intervals. Since these confidence
intervals and tests are not adjusted for multiple
comparisons, the corresponding model-wise findings
are interpreted as individual-comparison evidence
rather than simultaneous statistical guarantees
across all models and solver pairs.

\begin{table}[t]
	\centering
	\caption{Model-wise paired held-out cost differences over the eleven common
		seeds. Differences are computed as the first-named solver minus the second-named solver. W/T/L counts refer to wins/ties/losses of the first-named solver, where a win means lower cost.}
	\label{tab:three-solver-paired}
	\scriptsize
	\resizebox{\linewidth}{!}{%
		\begin{tabular}{@{}llrrrr@{}}
			\toprule
			Model & Comparison & Mean diff. & 95\% CI & W/T/L & Sign $p$ \\
			\midrule
			\texttt{TurtleBot~4} & \texttt{NMLS}$-$\texttt{MA-ES} & -0.5038 & $[-4.0648,3.0572]$ & 5/0/6 & 1.0000 \\
			& \texttt{VRBBO}$-$\texttt{MA-ES} & 4.5019 & $[-0.0025,9.0062]$ & 1/0/10 & 0.0117 \\
			& \texttt{VRBBO}$-$\texttt{NMLS} & 5.0056 & $[-0.3151,10.3263]$ & 2/0/9 & 0.0654 \\
			\addlinespace
			\texttt{Jackal-like} & \texttt{NMLS}$-$\texttt{MA-ES} & -0.1155 & $[-3.9407,3.7097]$ & 6/2/3 & 0.5078 \\
			& \texttt{VRBBO}$-$\texttt{MA-ES} & 2.8490 & $[-0.9823,6.6804]$ & 4/0/7 & 0.5488 \\
			& \texttt{VRBBO}$-$\texttt{NMLS} & 2.9645 & $[0.0544,5.8747]$ & 3/0/8 & 0.2266 \\
			\addlinespace
			\texttt{Ackermann} & \texttt{NMLS}$-$\texttt{MA-ES} & -0.2094 & $[-2.4735,2.0548]$ & 5/0/6 & 1.0000 \\
			& \texttt{VRBBO}$-$\texttt{MA-ES} & 4.0527 & $[-0.8067,8.9121]$ & 3/0/8 & 0.2266 \\
			& \texttt{VRBBO}$-$\texttt{NMLS} & 4.2621 & $[0.5446,7.9796]$ & 3/0/8 & 0.2266 \\
			\bottomrule
	\end{tabular}}
\end{table}

Pooling the 33 model--seed pairs descriptively gives the summary in
Table~\ref{tab:three-solver-pooled}. \texttt{NMLS} has the best pooled mean, median,
best, worst, failure count, and validation-acceptance count. Standard
\texttt{MA-ES} has the smallest pooled across-seed standard deviation,
$3.5197$ versus $3.8350$ for \texttt{NMLS} and $4.8955$ for \texttt{VRBBO}. \texttt{NMLS} reduces the
pooled mean by $0.55\%$ relative to standard \texttt{MA-ES}. \texttt{VRBBO} has
$7.54\%$ higher pooled mean cost than standard \texttt{MA-ES} and $8.13\%$
higher pooled mean cost than \texttt{NMLS}.

\begin{table}[t]
	\centering
	\caption{Descriptive pooled statistics over all 33 model--seed pairs. The
		pooling is used as a cross-model descriptive summary, not as a substitute for
		the model-wise paired analyses.}
	\label{tab:three-solver-pooled}
	\small
	\begin{tabular}{@{}lrrrrrrr@{}}
		\toprule
		Solver & Mean & Median & Best & Worst & Std. & Fail. & Acc. \\
		\midrule
		\texttt{MA-ES} & 50.4426 & 49.8447 & 45.8879 & 60.2530 & \textbf{3.5197} & 49 & 19/33 \\
		\texttt{MA-ES}+\texttt{NMLS} & \textbf{50.1664} & \textbf{49.0934} & \textbf{45.8598} & \textbf{60.0585} & 3.8350 & \textbf{26} & \textbf{28/33} \\
		\texttt{VRBBO} & 54.2438 & 54.2357 & 46.2714 & 67.3938 & 4.8955 & 49 & 21/33 \\
		\bottomrule
	\end{tabular}
\end{table}

The pooled paired \texttt{NMLS}-minus-\texttt{MA-ES} mean difference is $-0.2762$ with
95\% interval $[-1.9578,1.4053]$ and W/T/L count $16/2/15$. In contrast,
\texttt{VRBBO} minus standard \texttt{MA-ES} has mean difference $3.8012$ with 95\%
interval $[1.5290,6.0734]$ and W/T/L $8/0/25$, while \texttt{VRBBO} minus \texttt{NMLS} has
mean difference $4.0774$ with interval $[1.9568,6.1981]$ and W/T/L
$8/0/25$. Both pooled sign tests comparing \texttt{VRBBO} with a \texttt{MA-ES}
variant give $p=0.0046$. These cross-model summaries support the conclusion
that, under the common 432-trial budget, both matrix-adaptation variants are
more effective than \texttt{VRBBO} on this benchmark family. They do not imply a
universal ranking outside the declared models, parameterization, and budget.

The risk-active final tests provide a second axis of comparison. Across
$33\times32=1056$ final trials per solver, standard \texttt{MA-ES}, \texttt{NMLS},
and \texttt{VRBBO} record 49, 26, and 49 failures, respectively. Thus \texttt{NMLS} has 46.9\%
fewer observed failures than either comparator. The fixed-sample validation
rule is satisfied in 19/33 standard-\texttt{MA-ES} runs, 28/33 \texttt{NMLS} runs,
and 21/33 \texttt{VRBBO} runs. These pooled counts are descriptive; the formal
validation statement remains run-specific because each optimized controller
has its own 32-trial final test.

\subsection{Graphical comparison}
\label{sec:graphical-comparison}

To reduce page count, related plots are grouped into multi-panel figures.
Figure~\ref{fig:stat-cost-family} compares three distributional measures for
the budget-matched eleven-seed benchmark. The mean panel shows the same
ordering on all three models: \texttt{NMLS} is slightly below standard
\texttt{MA-ES}, while \texttt{VRBBO} is clearly higher. The median panel shows
that the \texttt{NMLS} advantage is strongest on \texttt{TurtleBot~4} and
\texttt{Jackal-like}, whereas standard \texttt{MA-ES} has the lower
\texttt{Ackermann} median. The standard-deviation panel further shows that the
mean advantage should not be interpreted as uniform dominance: standard
\texttt{MA-ES} is less variable on \texttt{TurtleBot~4} and
\texttt{Jackal-like}, while \texttt{NMLS} is markedly less variable on
\texttt{Ackermann}. In all three panels, lower values are preferable.

Taken together, Figure~\ref{fig:stat-cost-family} shows that \texttt{NMLS}
attains the lowest mean held-out cost on every vehicle model, whereas the
across-seed variability advantage depends on the model. The horizontal axis in
each panel lists \texttt{TurtleBot~4}, \texttt{Jackal-like}, and
\texttt{Ackermann}; within each model group, the left, middle, and right bars
correspond to standard \texttt{MA-ES}, \texttt{MA-ES+NMLS}, and
\texttt{VRBBO}, respectively. The three vertical coordinates are the mean
held-out cost, median held-out cost, and across-seed sample standard deviation.

The corresponding model-specific per-seed comparisons are collected in
Figure~\ref{fig:stat-seed-family}. In every panel, the horizontal axis is the
optimization seed and the vertical axis is the independent-test mean cost of
the controller produced by that seed; lower values are preferable. Standard
\texttt{MA-ES}, \texttt{MA-ES+NMLS}, and \texttt{VRBBO} are represented,
respectively, by the blue solid curve with circle markers, the orange dashed
curve with square markers, and the green dotted curve with triangle markers.
The frequent crossings between the two \texttt{MA-ES} variants expose
optimizer randomness and are consistent with their model-wise paired
confidence intervals containing zero. \texttt{VRBBO} is occasionally
competitive on individual seeds, but its curve lies above one or both
matrix-adaptation variants for most seeds; this is reflected by only eight
\texttt{VRBBO} wins in each pooled 33-pair comparison. These per-seed
variations also show why no conclusion should be based on a single favorable
optimization seed.

\begin{figure}[!http]
	\centering
	\begin{minipage}[t]{0.32\linewidth}
		\centering
		\includegraphics[width=\linewidth]{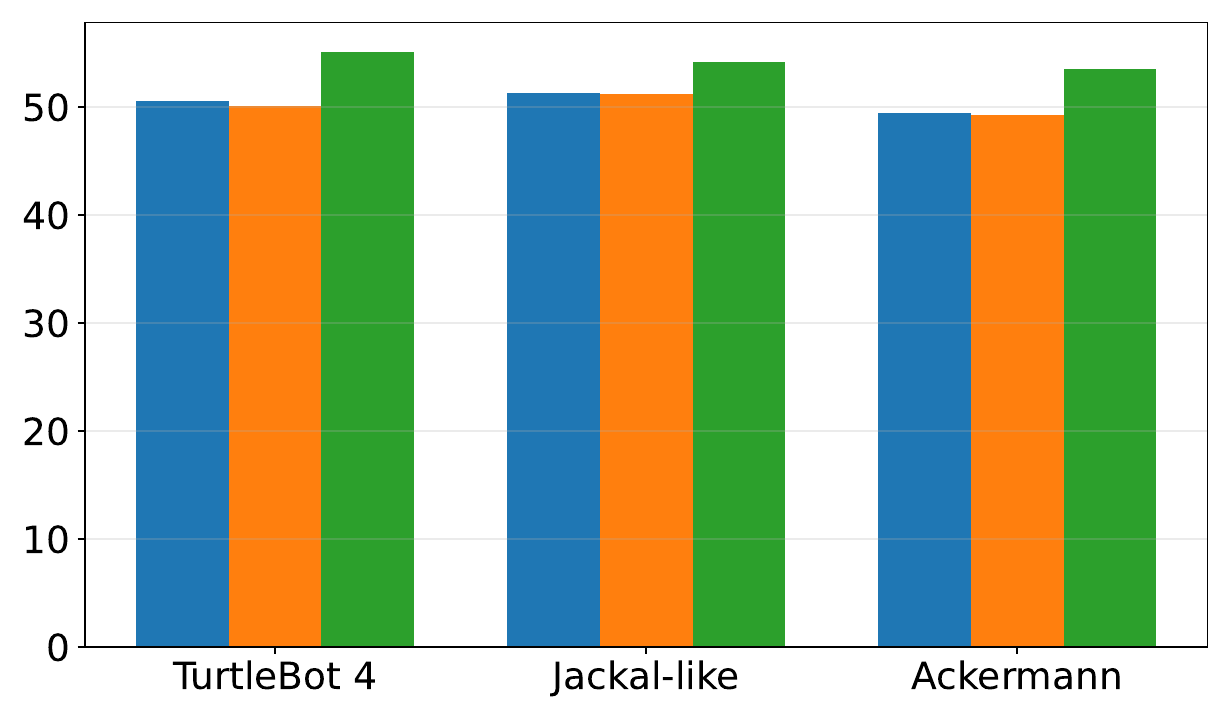}\\[-0.5ex]
		\footnotesize (a) Mean held-out cost
	\end{minipage}\hfill
	\begin{minipage}[t]{0.32\linewidth}
		\centering
		\includegraphics[width=\linewidth]{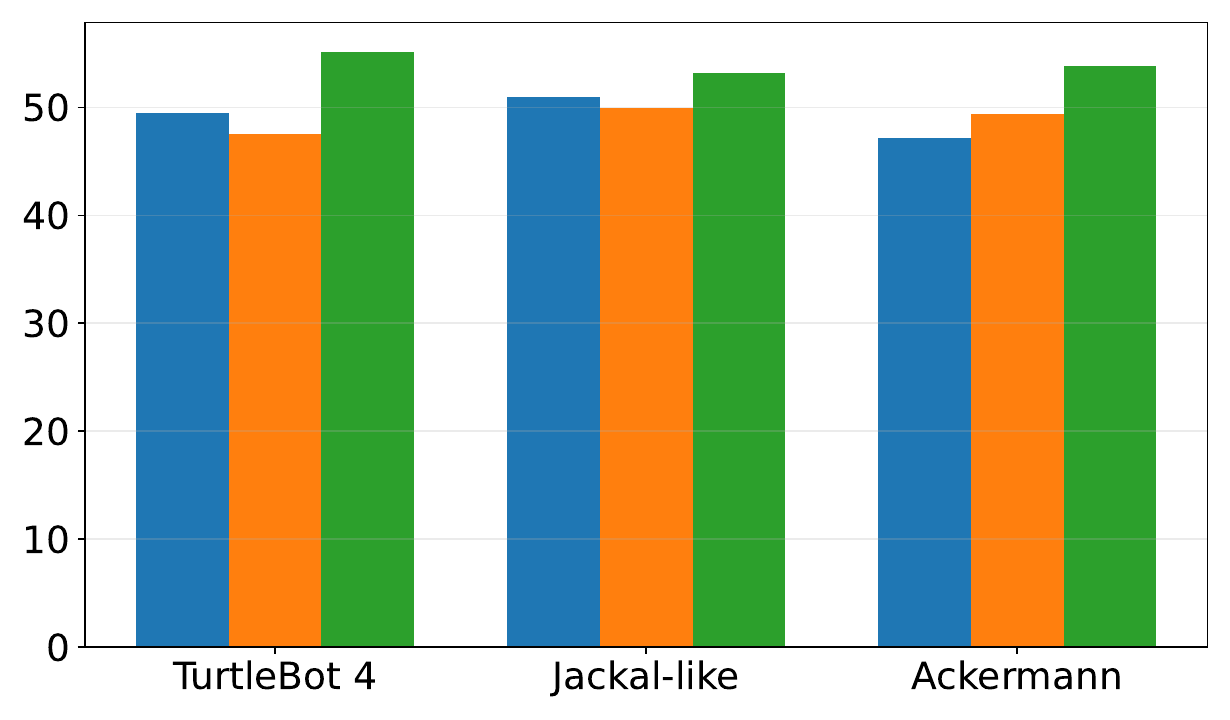}\\[-0.5ex]
		\footnotesize (b) Median held-out cost
	\end{minipage}\hfill
	\begin{minipage}[t]{0.32\linewidth}
		\centering
		\includegraphics[width=\linewidth]{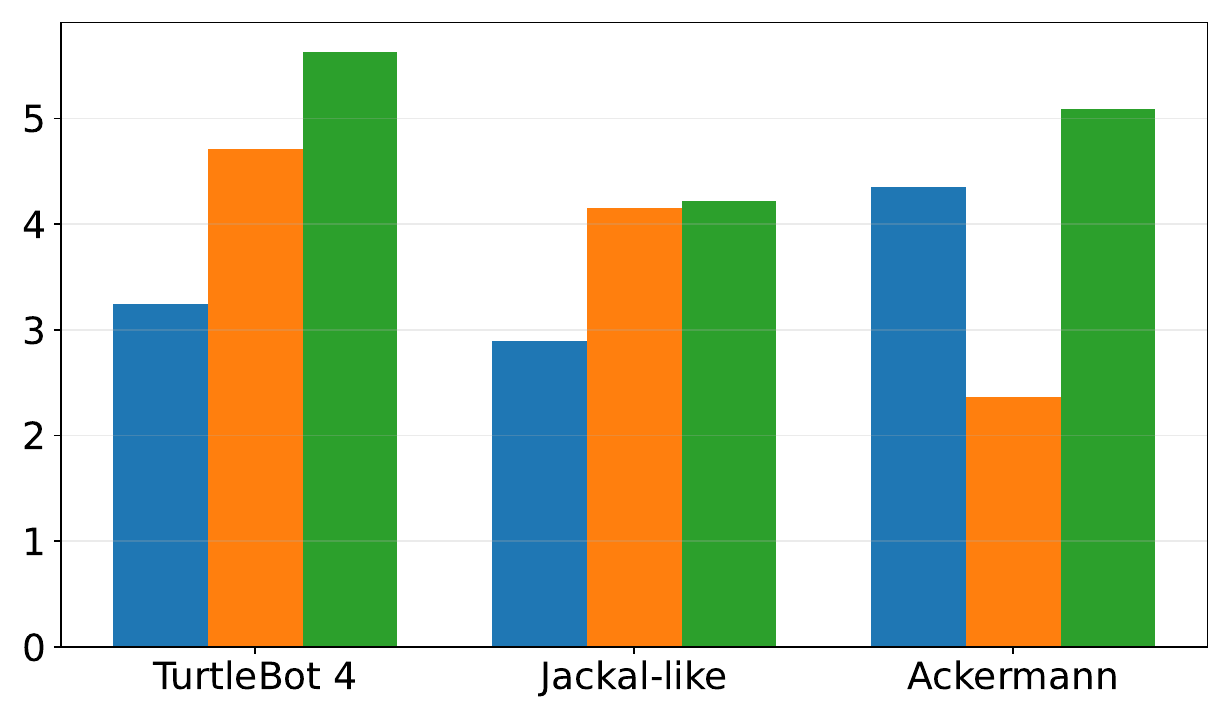}\\[-0.5ex]
		\footnotesize (c) Across-seed standard deviation
	\end{minipage}
	\caption{Cost-distribution statistics for the budget-matched eleven-seed
		benchmark: (a) mean held-out cost, (b) median held-out cost, and
		(c) across-seed sample standard deviation. Within each model group,
		the bars from left to right denote standard \texttt{MA-ES},
		\texttt{MA-ES+NMLS}, and \texttt{VRBBO}.}
	\label{fig:stat-cost-family}
\end{figure}

\begin{figure}[!http]
	\centering
	\begin{minipage}[t]{0.32\linewidth}
		\centering
		\includegraphics[width=\linewidth]{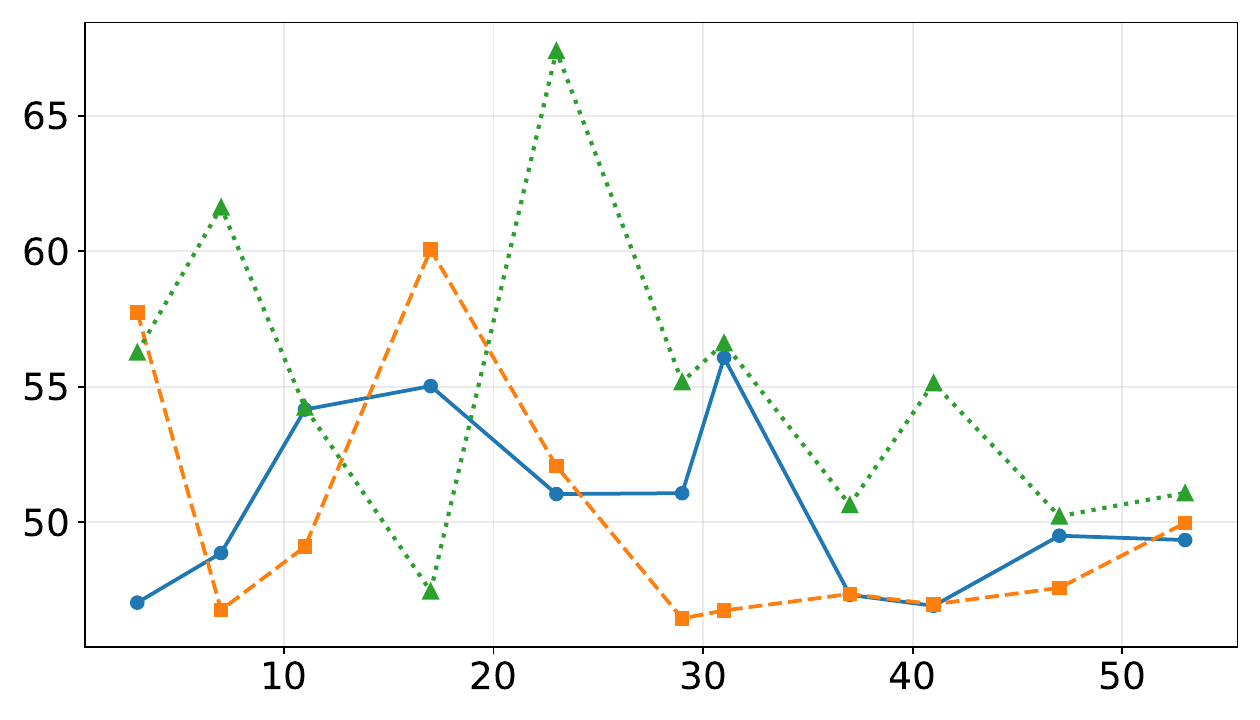}\\[-0.5ex]
		\footnotesize (a) \texttt{TurtleBot~4}
	\end{minipage}\hfill
	\begin{minipage}[t]{0.32\linewidth}
		\centering
		\includegraphics[width=\linewidth]{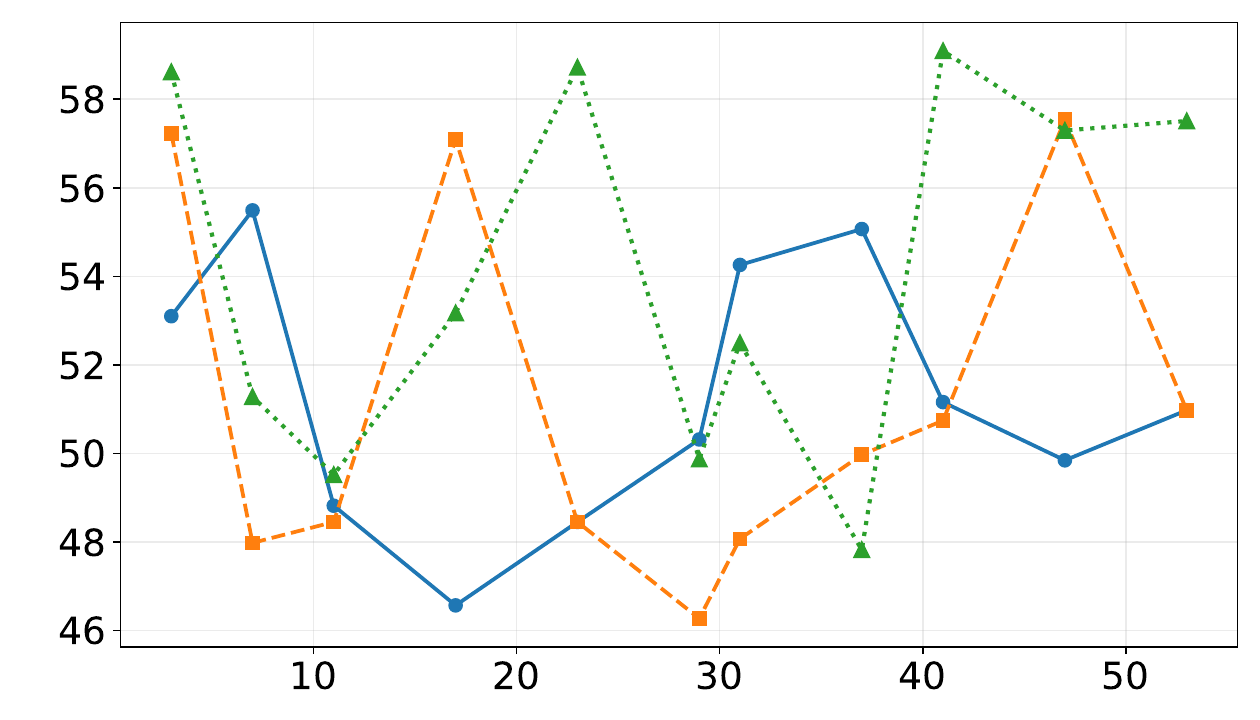}\\[-0.5ex]
		\footnotesize (b) \texttt{Jackal-like}
	\end{minipage}\hfill
	\begin{minipage}[t]{0.32\linewidth}
		\centering
		\includegraphics[width=\linewidth]{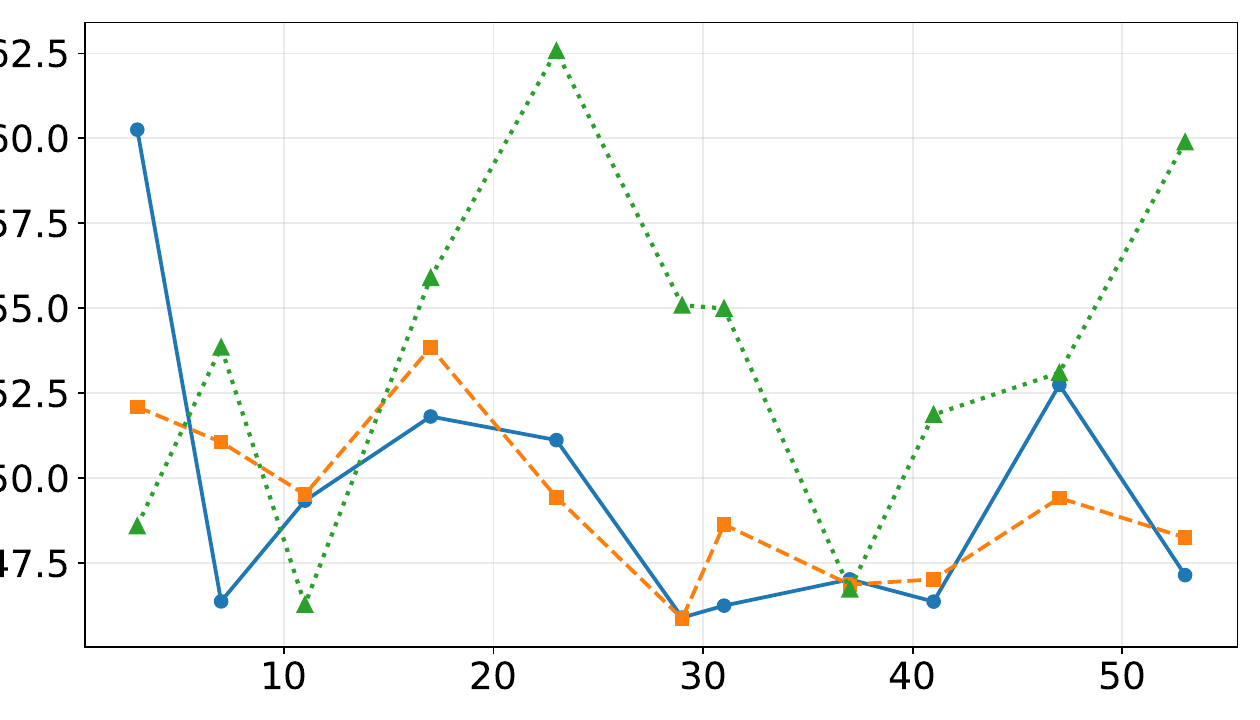}\\[-0.5ex]
		\footnotesize (c) \texttt{Ackermann}
	\end{minipage}
	\caption{Per-seed held-out-cost comparisons for
		(a) \texttt{TurtleBot~4}, (b) \texttt{Jackal-like}, and
		(c) \texttt{Ackermann}. The curves denote standard \texttt{MA-ES},
		\texttt{MA-ES+NMLS}, and \texttt{VRBBO}, respectively.}
	\label{fig:stat-seed-family}
\end{figure}

Figure~\ref{fig:stat-risk-family} groups the risk and pooled-order summaries.
In panels~(a) and~(b), the horizontal axis lists \texttt{TurtleBot~4},
\texttt{Jackal-like}, and \texttt{Ackermann}, with the left, middle, and right
bars corresponding to standard \texttt{MA-ES}, \texttt{MA-ES+NMLS}, and
\texttt{VRBBO}, respectively. Panel~(a) reports the total number of final-test
failures accumulated over the eleven independently optimized controllers for
each model, so lower values are better; \texttt{NMLS} has fewer observed
failures on every model. Panel~(b) reports the number of runs, out of 11,
satisfying the fixed-sample validation rule, so higher values are better;
\texttt{NMLS} gives the largest number of validation-accepted runs on every
model. In panel~(c), the horizontal axis is the pooled order statistic from
1 (best) to 33 (worst) across all model--seed pairs and the vertical axis is
the independent-test mean cost. The blue solid, orange dashed, and green
dotted curves denote standard \texttt{MA-ES}, \texttt{MA-ES+NMLS}, and
\texttt{VRBBO}, respectively. The sorted pooled-cost curves show a similar
performance ordering over most of the distribution and a noticeably heavier
upper tail for \texttt{VRBBO}. The small separation between the two
\texttt{MA-ES} curves is consistent with the nonsignificant
\texttt{NMLS}-versus-\texttt{MA-ES} paired interval, whereas the larger
\texttt{VRBBO} separation is consistent with the positive pooled paired
intervals reported above. Thus, the risk and pooled-cost panels consistently
favor \texttt{NMLS} on the observed failure and validation measures while
showing the larger upper-tail cost of \texttt{VRBBO} under the matched budget.

\begin{figure}[!http]
	\centering
	\begin{minipage}[t]{0.32\linewidth}
		\centering
		\includegraphics[width=\linewidth]{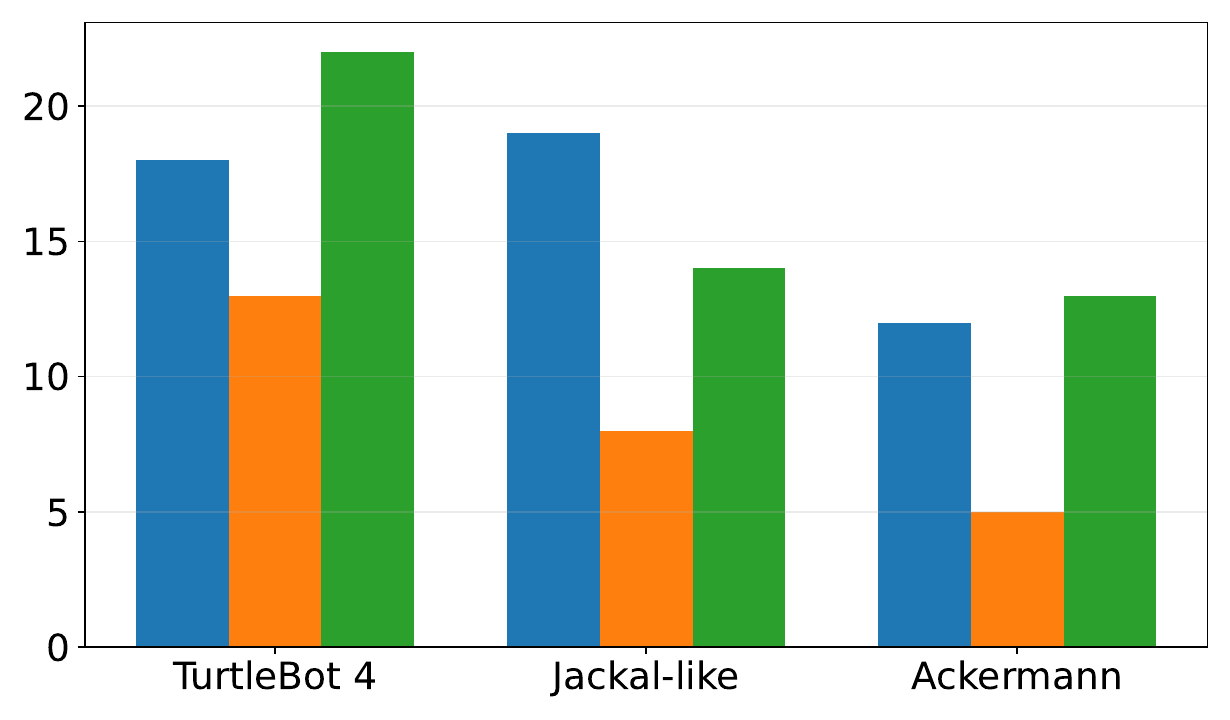}\\[-0.5ex]
		\footnotesize (a) Final-test failures
	\end{minipage}\hfill
	\begin{minipage}[t]{0.32\linewidth}
		\centering
		\includegraphics[width=\linewidth]{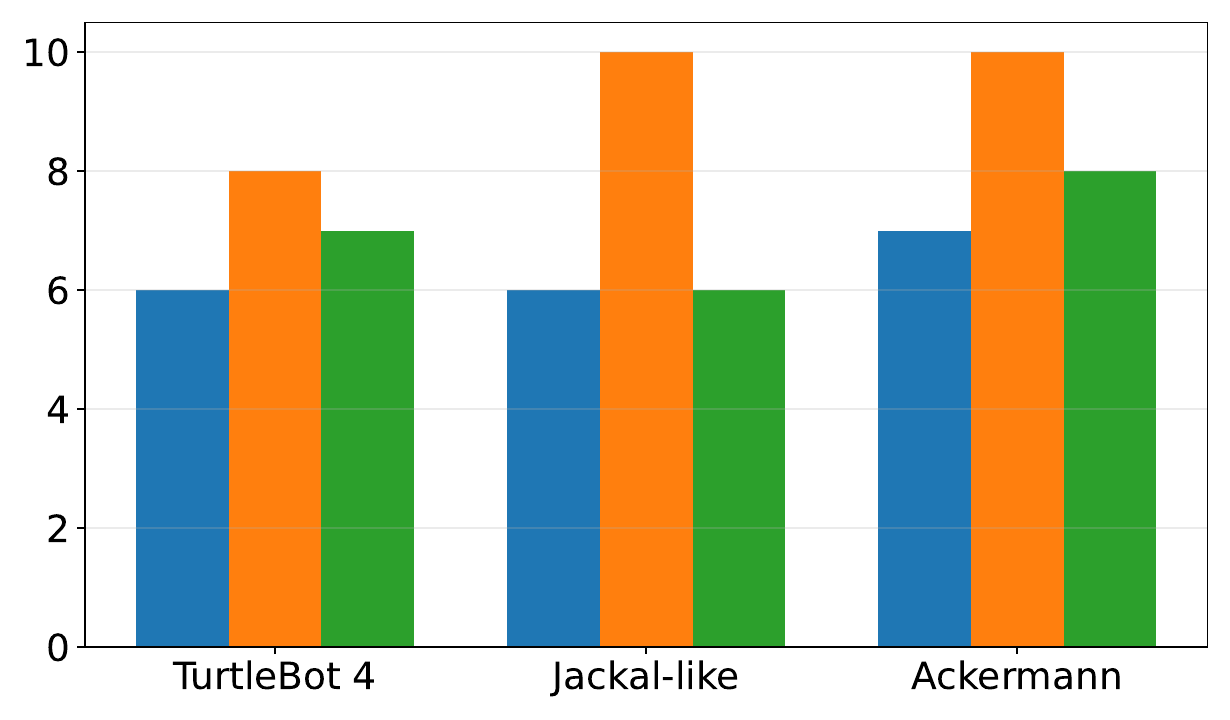}\\[-0.5ex]
		\footnotesize (b) Validation-accepted runs
	\end{minipage}\hfill
	\begin{minipage}[t]{0.32\linewidth}
		\centering
		\includegraphics[width=\linewidth]{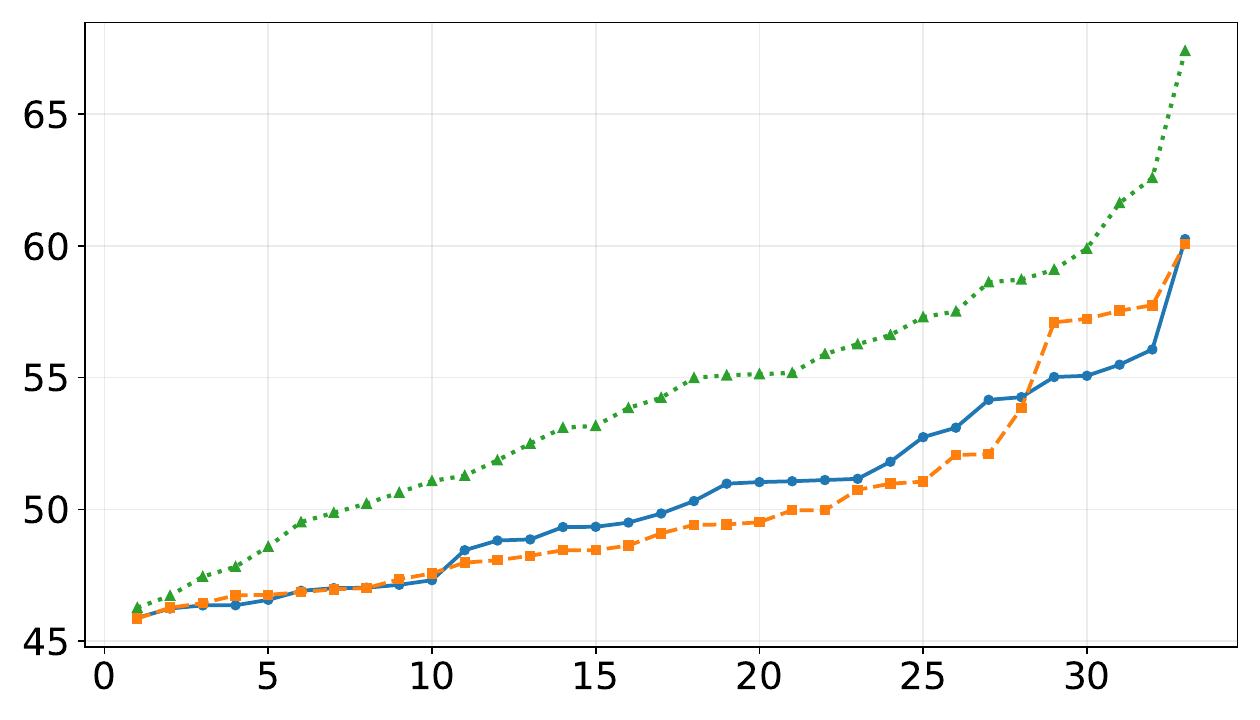}\\[-0.5ex]
		\footnotesize (c) Pooled sorted costs
	\end{minipage}
	\caption{Risk and pooled-performance summaries:
		(a) final-test failures, (b) validation-accepted runs, and
		(c) sorted held-out costs pooled over the 33 model--seed pairs.
		In panels~(a) and~(b), bars from left to right denote standard
		\texttt{MA-ES}, \texttt{MA-ES+NMLS}, and \texttt{VRBBO}; the same
		solver ordering is used for the three curves in panel~(c).}
	\label{fig:stat-risk-family}
\end{figure}

\clearpage

\subsection{Measure-wise recommendation}
\label{sec:solver-recommendation}

The numerical evidence does not support a single solver being best for every
possible criterion. The recommendation in Figure~\ref{fig:recommendation}
therefore separates performance/risk, variability, and the overall
budget-matched choice. Under the present benchmark, \texttt{MA-ES}+\texttt{NMLS} is the
preferred default because it has the best pooled mean, median, best and worst
costs, the fewest observed failures, and the largest validation-acceptance
count. Standard \texttt{MA-ES} is the preferred alternative when the primary
criterion is pooled across-seed standard deviation, where it is the most
stable of the three. \texttt{VRBBO} is not preferred under the 432-trial budget: it has
the highest pooled mean, median, worst cost, and standard deviation, and the
pooled paired differences favor both matrix-adaptation variants. This is a
budget- and benchmark-specific recommendation, not a universal ranking of
black-box optimization methods.

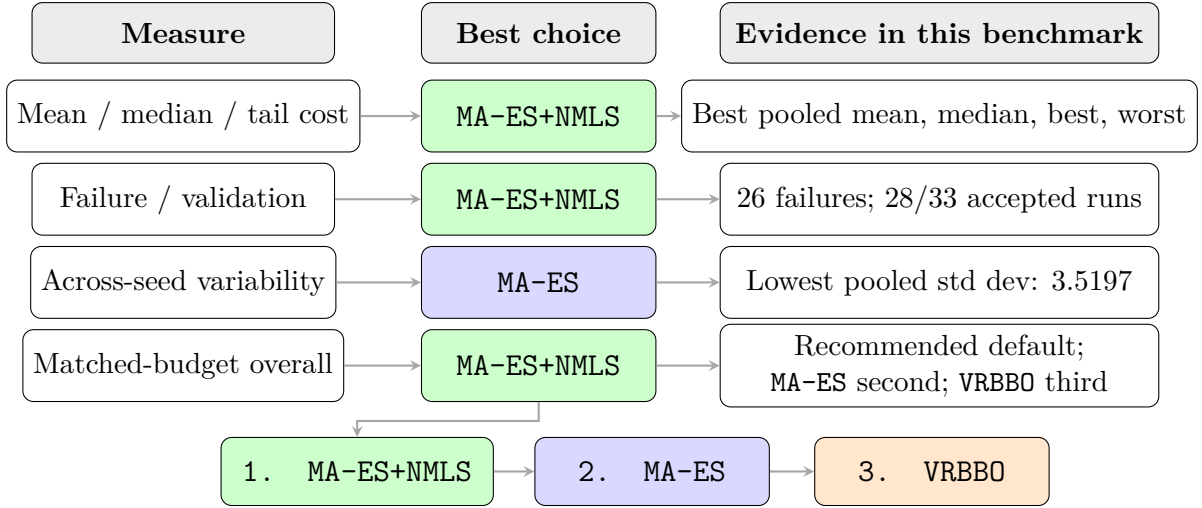
\begin{figure}[t]
	\centering
	\begin{tikzpicture}[x=1cm, y=1cm, font=\small, >=stealth]
		\tikzset{
			header/.style={draw, rounded corners=4pt, fill=gray!15, font=\bfseries\small, minimum width=#1, minimum height=0.8cm},
			cell/.style={draw, rounded corners=4pt, minimum width=#1, minimum height=0.95cm, align=center},
			algogreen/.style={cell=3.1cm, fill=green!20, font=\ttfamily},
			algoblue/.style={cell=3.1cm, fill=blue!15, font=\ttfamily},
			algoorange/.style={cell=3.1cm, fill=orange!20, font=\ttfamily},
			flow/.style={->, thick, gray!70}
		}
		
		\node[header=4.0cm] (h1) at (-1,3.0) {Measure};
		\node[header=3.1cm] (h2) at (3.7,3.0) {Best choice};
		\node[header=5.8cm] (h3) at (9,3.0) {Evidence in this benchmark};
		
		\node[cell=4.0cm] (r1c1) at (-1,1.9) {Mean / median / tail cost};
		\node[algogreen] (r1c2) at (3.7,1.9) {MA-ES+NMLS};
		\node[cell=5.8cm] (r1c3) at (9,1.9) {Best pooled mean, median, best, worst};
		
		\node[cell=4.0cm] (r2c1) at (-1,0.8) {Failure / validation};
		\node[algogreen] (r2c2) at (3.7,0.8) {MA-ES+NMLS};
		\node[cell=5.8cm] (r2c3) at (9,0.8) {26 failures; 28/33 accepted runs};
		
		\node[cell=4.0cm] (r3c1) at (-1,-0.3) {Across-seed variability};
		\node[algoblue] (r3c2) at (3.7,-0.3) {MA-ES};
		\node[cell=5.8cm] (r3c3) at (9,-0.3) {Lowest pooled std dev: 3.5197};
		
		\node[cell=4.0cm] (r4c1) at (-1,-1.4) {Matched-budget overall};
		\node[algogreen] (r4c2) at (3.7,-1.4) {MA-ES+NMLS};
		\node[cell=5.8cm] (r4c3) at (9,-1.4) {Recommended default;\\ \texttt{MA-ES} second; \texttt{VRBBO} third};
		
		\node[algogreen, minimum width=3.6cm, minimum height=0.9cm] (f1) at (1.3,-2.8) {1. MA-ES+NMLS};
		\node[algoblue, minimum height=0.9cm] (f2) at (5.2,-2.8) {2. MA-ES};
		\node[algoorange, minimum height=0.9cm] (f3) at (8.9,-2.8) {3. VRBBO};
		
		\draw[flow] (r1c1) -- (r1c2); \draw[flow] (r1c2) -- (r1c3);
		\draw[flow] (r2c1) -- (r2c2); \draw[flow] (r2c2) -- (r2c3);
		\draw[flow] (r3c1) -- (r3c2); \draw[flow] (r3c2) -- (r3c3);
		\draw[flow] (r4c1) -- (r4c2); \draw[flow] (r4c2) -- (r4c3);
		
		\draw[flow] (r4c2.south) -- ++(0,-0.25) -| (f1.north);
		
		\draw[flow] (f1) -- (f2);
		\draw[flow] (f2) -- (f3);
		
	\end{tikzpicture}
	\caption{Measure-wise recommendation for the declared 432-complete-trial benchmark. The overall ordering summarizes the present data only; it is not a universal solver ranking.}
	\label{fig:recommendation}
\end{figure}
\subsection{Discussion and limitations}

The new three-solver benchmark strengthens two conclusions while preserving
the earlier caveats. First, the frozen \texttt{NMLS} parameters transfer without
model-specific retuning: \texttt{NMLS} has the lowest mean held-out cost on \texttt{TurtleBot~4},
\texttt{Jackal-like}, and \texttt{Ackermann}. Second, the comparison with \texttt{VRBBO} shows that the
matrix-adaptation population mechanism is particularly effective under the
present small complete-trial budget. Pooled across the 33 matched model--seed
pairs, the reported paired intervals exclude zero in favor of both
\texttt{MA-ES} variants relative to \texttt{VRBBO}, and \texttt{VRBBO} wins only eight of
33 paired comparisons against either variant. Because this pooling combines
three benchmark families, we treat it as a cross-model descriptive summary
rather than as a universal inferential statement. The result therefore provides a
concrete numerical reason to retain \texttt{MA-ES} as the outer-loop framework
for this application.

The evidence for \texttt{NMLS} relative to standard \texttt{MA-ES} is more nuanced.
\texttt{NMLS} leads the pooled performance and validation measures and has lower mean
cost on every model, but its pooled paired interval versus standard
\texttt{MA-ES} contains zero and the standard method has the smallest pooled
across-seed standard deviation. The appropriate statement is therefore that
\texttt{NMLS} is the recommended default under the present multi-measure benchmark,
not that it is uniformly or statistically significantly superior to standard
\texttt{MA-ES} on every criterion.

The quick update-aware experiment and the exact-budget optimizer benchmark
measure different aspects of the method. The quick study evaluates the
confidence-supported racing mechanism but is too small for any generation to
be certified. The exact-budget three-solver benchmark is risk-active and
compares optimizer behavior fairly, but it uses fixed-$K$ empirical
feasibility-first evaluation to isolate the outer search. A future experiment
should combine a risk-active benchmark, a racing budget large enough to permit
certification, and the three-solver comparison under a common total simulation
budget. Hardware dynamics, latency, state-estimation dynamics, distribution
shift, global optimality, and closed-loop stability are not established by the
present kinematic simulations.


\section{Conclusion and limitations}
\label{sec:conclusion}

We developed a risk-constrained, update-aware \texttt{MA-ES}
framework for tuning \texttt{MPPI} controllers from noisy
closed-loop evaluations. The proposed approach connects
statistical confidence in population rankings with explicit
bounds on the evolutionary updates, incorporates a
derivative-free nonmonotone line search when the
recombination direction is certified, and independently
validates the final controller. A worst-case
arithmetic-complexity analysis also characterizes the
computational requirements of online control and offline
tuning.

Numerical experiments on three kinematic vehicle models
demonstrate the transferability of the proposed
\texttt{NMLS} refinement without model-specific retuning.
Under a common evaluation budget, \texttt{MA-ES+NMLS}
achieves lower mean held-out costs than standard
\texttt{MA-ES} on all three models, reduces the pooled
mean cost from $50.4426$ to $50.1664$, and increases
the number of validation-accepted runs from $19/33$
to $28/33$. Both \texttt{MA-ES} variants also achieve
lower pooled mean costs than the tested \texttt{VRBBO}
configuration. These findings support the practical
competitiveness and transferability of the framework,
although the observed \texttt{NMLS} cost improvements
over standard \texttt{MA-ES} are not statistically
significant under the reported paired comparisons.

The experiments also reveal an important limitation
of the statistical certification mechanism. Under
the small racing budgets used in the update-aware study,
no generation satisfies the rank-certification condition.
Consequently, the available experiments establish the
executability of the complete framework and provide
evidence for the effectiveness of the nonmonotone
refinement, but do not yet demonstrate the practical
benefit of certificate-triggered line-search activation.

Future work should investigate larger racing budgets,
risk-active environments, more challenging disturbances,
and additional robotic platforms. Hardware experiments
and real-time implementation studies are also needed,
as the present results do not establish closed-loop
stability, global optimality, or general safety guarantees.

\section*{Funding}

M.K. acknowledges financial support of the Austrian Science Foundation under \url{https://doi.org/10.55776/PAT2747625}.



\end{document}